\documentclass[pdflatex,sn-mathphys,Numbered]{sn-jnl}
\usepackage{amssymb}
\usepackage{graphicx}
\usepackage{amscd}
\usepackage{amsmath}
\usepackage{tikz}
\usepackage{xcolor}
\usepackage{accents}
\usepackage{hyperref} 
\usepackage{multirow}%
\usepackage{amsfonts}%
\usepackage{amsthm}%
\usepackage{mathrsfs}%
\usepackage[title]{appendix}%
\usepackage{textcomp}%
\usepackage{manyfoot}%
\usepackage{booktabs}%
\usepackage{algorithm}%
\usepackage{algorithmicx}%
\usepackage{algpseudocode}%
\usepackage{listings}%

\newcommand\tr{\operatorname{tr}}
\newcommand\skw{\operatorname{skw}}

\renewcommand\div{\operatorname{div}}
\newcommand\curl{\operatorname{curl}}

\newcommand\diam{\operatorname{diam}}

\newcommand\Span{\operatorname{span}}
\newcommand\supp{\operatorname{supp}}

\newcommand\vol{\mathsf{vol}}

\newcommand\R{\mathbb{R}}

\newcommand\B{{\mathcal B}}
\newcommand\C{{\mathscr C}}

\newcommand\I{{\mathcal I}}

\renewcommand\P{{\mathcal P}}
\newcommand\Q{\mathcal Q}

\newcommand\M{{\mathcal M}}
\newcommand\T{{\mathcal T}}
\renewcommand\S{{\mathcal S}}

\renewcommand{\>}{\rangle}
\newcommand{\0}{\mathaccent23}

\newcommand\e{\mathbf e}
\newcommand\Ball{\mathfrak{B}}

\newcommand\sig{\partial}

\newcommand\Alt{\operatorname{Alt}}
\DeclareMathOperator{\sign}{sign}

\newcommand{\circo}{~{\tikz \draw[line width=0.2pt] circle(0.8pt);}~}
\newcommand{\dzero}[1]{\accentset{\circo\mkern-15mu\circo}{#1}}

\numberwithin{equation}{section}
\newtheorem{thm}{Theorem}[section]
\newtheorem{prop}[thm]{Proposition}
\newtheorem{lem}[thm]{Lemma}
\newtheorem{cor}[thm]{Corollary}
\newtheorem{remark}{Remark}

\begin{document}

\title[The bubble transform]
      {Local space-preserving decompositions for the bubble transform}

\author*[1]{\fnm{Richard} \sur{Falk}}\email{falk@math.rutgers.edu}
\author[2]{\fnm{Ragnar} \sur{Winther}}
  \email{rwinther@math.uio.no}
  \affil*[1]{\orgdiv{Department of Mathematics}, \orgname{Rutgers University},
    \orgaddress{\street{} \city{Piscataway}, \postcode{08854},
      \state{New Jersey}, \country{USA}}}
  \affil[2]{\orgdiv{Department of Mathematics}, \orgname{University of Oslo},
    \orgaddress{\street{} \city{Oslo}, \postcode{0316}, \state{}
      \country{Norway}}}

  \abstract{
  The bubble transform is a procedure to decompose
      differential forms, which are piecewise smooth with respect to a
      given triangulation of the domain, into a sum of local bubbles.
      In this paper, an improved version of a construction in the
      setting of the de Rham complex previously proposed by the
      authors is presented.  The major improvement in the
      decomposition is that unlike the previous results, in which the
      individual bubbles were rational functions with the property
      that groups of local bubbles summed up to preserve piecewise
      smoothness, the new decomposition is strictly space-preserving
      in the sense that each local bubble preserves piecewise
      smoothness.  An important property of the transform is that the
      construction only depends on the given triangulation of the
      domain and is independent of any finite element space. On the
      other hand, all the standard piecewise polynomial spaces are
      invariant under the transform.  Other key properties of the
      transform are that it commutes with the exterior derivative, is
      bounded in $L^2$, and satisfies the {\it stable decomposition
        property}. 
   }

\pacs[MSC Classification]{65N30, 52-08
  \\
  \\
 ------------------------------------
 \\
 Communicated by Snorre Christiansen
\\
 ------------------------------------}

\date{January 23, 2025}
       
\maketitle 

\section{Introduction}\label{sec:intro}
The present paper is a continuation of the earlier papers
\cite{bubble-I} and \cite{bubble-II} of the authors.  While the first
paper was devoted to scalar valued functions, the second paper,
\cite{bubble-II}, develops a theory for decomposing differential forms
into a sum of functions, or bubbles, {which have local support on
  domains defined by a given simplicial mesh of the domain. The
  decomposition, which we refer to as {\it the bubble transform},
  commutes with the exterior derivative, and has the additional
  property that all the standard piecewise polynomial spaces utilized
  in the finite element exterior calculus (FEEC),
  cf. \cite{FEEC-book,acta,bulletin}, are in some sense invariant.
  However, the piecewise polynomial spaces are not {\it strictly}
  invariant for the decomposition constructed in \cite{bubble-II}. In
  general, each individual bubble is a rational function, but with the
  property that groups of the local bubbles sum up to preserve the
  desired polynomial structure.  The purpose of the theory presented
  here is to refine the earlier theory, so that we obtain a transform
  which is strictly {\it space-preserving} in the sense that each
  local bubble preserves piecewise smoothness and the standard
  piecewise polynomial spaces of FEEC.

The development of the bubble transform is partly motivated by the
$hp$-finite element method, i.e., where both piecewise polynomials of
arbitrary high degree and arbitrary small mesh cells are allowed.
While the analysis of finite element methods based on mesh refinements
and a fixed polynomial degree, i.e., the $h$-method, is by now well
understood, the corresponding analysis for the $p$-method, where the
polynomial degree is unbounded, is so far less canonical.  However,
the bubble transform represents a theory where the decomposition
itself, and the associated operator bounds, are obtained independently
of any finite element space. The entire construction only depends on a
given triangulation of the domain.  In fact, the decomposition is also
stable with respect to mesh refinements, and therefore the results
will apply to general $hp$-methods.  As a consequence, the
decomposition represents a new tool for understanding $hp$-methods. As
an example, consider the analysis of overlapping Schwarz
preconditioners. In \cite{additive-schwarz}, it is established how to
construct such preconditioners for second order elliptic problems in
the setting of $hp$-refinements.  But so far, the corresponding
verification for more general Hodge-Laplace problems appears to be
open.  In fact, the key obstacle for establishing such a bound is to
verify the so-called stable decomposition property, i.e., to establish
the existence of a bounded decomposition, cf.  \cite[Chapter
  2]{tosseli-widlund} and references given therein. Such a bound is
simply a special case of the bounds we derive here.  Although the
discussion in the present paper will be restricted to the basic theory
of the bubble transform, a more thorough motivation can be found in
\cite{bubble-II}.

In order to describe the main results of this paper, it is first
necessary to introduce some basic notation.  Throughout this paper,
$\Omega$ will be a bounded polyhedral domain in $\R^n$, and for $0 \le
k \le n$, the space of smooth differential $k$-forms on $\Omega$ will
be denoted $\Lambda^k(\Omega)$.  The construction of the bubble
transform is based on a simplicial mesh $\T$ of $\Omega$.  The
corresponding space, $\Lambda^k(\T)$, is the space of $k$--forms on
$\Omega$ which are piecewise smooth with respect to $\T$. More
precisely, the elements of $\Lambda^k(\T)$ are smooth on the closed
simplices $T$ in the triangulation and have single-valued traces on
each subsimplex of $\T$.  For the piecewise polynomial subspaces of
$\Lambda^k(\T)$, we will adopt the standard notation of FEEC, cf.
\cite{FEEC-book,acta,bulletin}, i.e., $\P_r\Lambda^k(\T)$ denotes the
space of piecewise polynomial forms of degree less than or equal to
$r$, while $\P_r^-\Lambda^k(\T)$ is the corresponding space of trimmed
polynomials.  The set of all subsimplices of $\T$ is denoted
$\Delta(\T)$, while $\Delta_m(\T)$ is the set of simplices of
dimension $m$.  For each $f \in \Delta(\T)$, the macroelement
$\Omega_f$ consists of the union of all $n$--simplices in $\Delta(\T)$
containing $f$ as a subsimplex.  Furthermore, $\T_f$ is the
restriction of the mesh $\T$ to the macroelement $\Omega_f$, and $\0
\Lambda^k(\T_f)$ is the subspace of $\Lambda^k(\T)$ consisting of
forms which have support on $\Omega_f$, i.e., which vanish on $\Omega
\setminus \Omega_f$.

For given $u \in \Lambda^k(\T)$, the bubble transform leads to a
decomposition of the form
\begin{equation}
  \label{main-decomp}
  u = W^ku + \sum_{f \in \Delta(\T)}B_f^ku = W^ku
  + \sum_{m=0}^n\sum_{f \in \Delta_m(\T)}B_f^ku,
\end{equation}
 where the bubbles $B_f^k u$ belong to $\0 \Lambda^k(\T_f)$, and where
 $W^k u$ is a trimmed piecewise linear $k$-form.  More precisely, we
 will show how to construct linear operators $W^k : \Lambda^k(\T) \to
 \P_1^-\Lambda^k(\T)$ and local operators $B_f^k : \Lambda^k(\T) \to
 \0 \Lambda^k(\T_f)$ which commute with the exterior derivative $d$,
 i,e.,
\[
dW^k = W^{k+1}d, \quad \text{and } dB_f^k = B_f^{k+1}d, \quad 0 \le k \le n-1.
\]
In fact, if we let $\B^k$ denote the collection of all the operators
$\{B_f^k\}_{f \in \Delta(\T)}$, such that we can view
\[
\B^k : \Lambda^k(\T) \to \prod_{f \in \Delta(\T)}\0 \Lambda^k(\T_f)
:= \0\Lambda^{k}(\T, \Delta),
\]
we can summarize and state that the diagram 
\begin{equation*}
\begin{CD}
\Lambda^k(\T) @>d>> 
\Lambda^{k+1}(\T)\\
@VV (W^k, \B^k) V  @VV(W^{k+1},\B^{k+1}) V\\
 \P_1^-\Lambda^k(\T) \times \0\Lambda^k(\T, \Delta) @>d>>
 \P_1^-\Lambda^{k+1}(\T) \times \0\Lambda^{k+1}(\T, \Delta)
\end{CD}
\end{equation*}
commutes.  We will also show that the operators $W^k$ and $B_f^k$ can
be extended to bounded operators in $L^2$.  Furthermore, the
polynomial preserving property of the bubble transform can simply be
expressed by the fact that
\begin{equation}
  \label{space-preserve}
B_f^k (\P_r\Lambda^k(\T)) \subset \0\P_r\Lambda^k(\T_f)
\quad \text{and }
B_f^k (\P_r^-\Lambda^k(\T)) \subset \0\P_r^-\Lambda^k(\T_f)
\end{equation}
for all $f \in \Delta(\T)$ and $r\ge 1$, where $\0\P_r\Lambda^k(\T_f)
= \P_r\Lambda^k(\T) \cap \0 \Lambda^k(\T_f)$ and with corresponding
definition for the $\P_r^-$-spaces.

The individual bubbles, $B_f^k u$, introduced above, will not
correspond to the bubbles constructed in \cite{bubble-II}.  However, a
key part of the analysis given in \cite{bubble-II} is the study of a
family of trace preserving operators, $C_m^k :\Lambda^k(\T) \to
\Lambda^k(\T)$, where $0 \le m \le n-1$. These operators are
explicitly given in formula \eqref{Cm-rewritten} below.  A key
property of these operators is that if $f \in \Delta_m(\T)$, where
$m\ge k$, then
\begin{equation}
  \label{trace-preserve}
\tr_f C_m^k u = \tr_f u,
\end{equation}
cf. \cite[Lemma 2.2]{bubble-II}. Here $\tr_f$ is the trace operator.
Furthermore, these operators commute with the exterior derivative and
they preserve piecewise smoothness and the standard spaces of
piecewise polynomials, cf. \cite[Proposition 7.1]{bubble-II}.  The
latter property may seem surprising since formula \eqref{Cm-rewritten}
contains rational mesh dependent functions.  It was also shown how the
global functions $C_m^ku$ can be decomposed into a sum of local
bubbles, but these bubbles were rational functions leading to an
apparent defect of the theory of \cite{bubble-II}. However, in the
analysis below, where we will overcome the problems just mentioned,
the operators $C_m^k$ will still play an essential part. In fact, the
new decomposition of $u \in \Lambda^k(\T)$ will be initialized by
expressing $u$ as
\[
u = (u - C_{n-1}^k u) + C_{n-1}^k u.
\]
It follows from \eqref{trace-preserve} above that $u - C_{n-1}^k u$ is
a global piecewise smooth function which has zero trace on elements of
$\Delta_{n-1}(\T)$.  As a consequence, we can decompose this function
as $u - C_{n-1}^k u = \sum_{f \in \Delta_n(\T)} B_f^k u$, where each
local bubble, $B_f^k u := \tr_f (u - C_{n-1}^k u )$, is piecewise
smooth, and with support in $f = \Omega_f$, for $f \in \Delta_n(\T)$.
Furthermore, the operators $\{B_f^k \}_{f \in \Delta_n(\T)}$ will
commute with the exterior derivative and preserve the piecewise
polynomial spaces.  We can conclude that we have decomposed $u$ into
\begin{equation}
  \label{pre-decomp}
  u = \sum_{f \in \Delta_n(\T)} B_f^k u+ C_{n-1}^k u, \quad B_f^k u
  =  \tr_f (u - C_{n-1}^k u ),
\end{equation}
where the first part is a sum of local bubbles. The task for the rest
of the construction will be to show that the second part, the function
$C_{n-1}^k u$, also admits such a decomposition.  We will achieve this
by focusing on the differences $C_m^k - C_{m-1}^k$, and from an
alternative representation of this operator, we will obtain the
desired decomposition into local bubbles which preserve piecewise
smoothness.  This will also lead to an alternative proof of the
property that the operators $C_m^k$ preserve piecewise smoothness.  In
fact, even in the case of zero-forms, or scalar valued functions, this
development deviates from the earlier approaches used for example in
\cite{bubble-I} or \cite{additive-schwarz}, cf. the discussion given
in Section~\ref{sec:scalar} below.  To prepare for the corresponding
procedure in the general case of $k$-forms, we construct various local
functions only depending on the mesh $\T$.  The delicate recursive
construction of these mesh functions, given in
Section~\ref{sec:mesh-functions}, represents a completely new approach
as compared to the construction performed in \cite{bubble-II}. The
results of this section, which are derived by utilizing chains with
values in spaces of trimmed linear differential forms, can be seen as
the main new tool used to obtain the improved results of this paper.

The present paper is organized as follows. In the next section, we
list some assumptions we will make, recall some standard notation and
properties of differential forms and simplicial complexes, and define
some of the basic operators that we will use in the construction of
our decomposition. We end Section~\ref{sec:prelims} with an outline of
the construction we will present in the remainder of the paper, and
state the main results, cf. Theorem~\ref{thm:main}.  To motivate the
general theory, we discuss the decomposition in the case of
scalar-valued functions in Section~\ref{sec:scalar}, while
Section~\ref{sec:mesh-functions} is devoted to an analysis of the
local structure of the mesh. The result of this local analysis is used
to define families of order reduction operators in
Section~\ref{sec:order-reduct}, and these operators are then used to
define the local bubbles $B_f^k$, cf. Section~\ref{sec:Kmgk}.  In
Section~\ref{sec:fund-ident}, we focus on the operator $C_m^k -
C_{m-1}^k$, and show that this operator admits a decomposition into
local bubbles with desired properties.  By utilizing a telescoping
series argument,
we will then obtain a similar decomposition for the operator
$C_{n-1}^k$.  In Section~\ref{sec:bounds}, we show that all the
operators of the decomposition \eqref{main-decomp} are bounded in
$L^2$. In the final section, we give a more detailed description of
the connection of this paper to previous work in the area (e.g.,
\cite{babuska-suri}, \cite{bubble-I}, \cite{K-M-R}, \cite{munoz-sola},
\cite{additive-schwarz}).

\section{Preliminaries}
\label{sec:prelims} 
In this section, we introduce the basic assumptions, notation, and
concepts that will be used in the construction below, and give an
outline of the complete theory.

\subsection{Assumptions}
\label{sec:assumptions}
We assume that $\Omega \subset \R^n$ is a bounded polyhedral Lipschitz
domain which is partitioned into a finite set of $n$ simplices,
$\Delta_n(\T)$.  The simplicial mesh $\T$ is assumed to be a
simplicial decomposition of $\Omega$, i.e., the union of the simplices
in $\Delta_n(\T)$ is the closure of $\Omega$ and the intersection of
any two is either empty or a common subsimplex of each.  The set of
all such simplices of dimension $m$ is denoted $\Delta_m(\T)$, while
$\Delta(\T) = \bigcup_{0 \le m \le n} \Delta_m(\T)$.  Furthermore,
below we will frequently write $\Delta$ instead of $\Delta(\T)$ and
$\Delta_m$ instead of $\Delta_m(\T)$.}

\subsection{Notation}
\label{sec:notation}
The space of smooth differential $k$-forms on $\Omega$ will be denoted
$\Lambda^k(\Omega)$. More precisely, for each $x \in \Omega$, $u_x \in
\Alt^k$, where $\Alt^k$ is the space of alternating $k$ -linear maps
$\R^n \times \cdots \times \R^n \to \R$.  We recall that a projection,
$\skw$, mapping $k$-linear forms to $\Alt^k$, is given by
\[
(\skw u)(v_1, \ldots , v_k) = \frac{1}{k!}
\sum_{\sigma} \sign(\sigma) u(v_{\sigma(1)}, \ldots ,v_{\sigma(k)}),
\]
where the sum is over all permutations of $\{1, \ldots , k\}$. In our
discussion below, we will encounter both the tensor product,
$\otimes$, and the wedge product, $\wedge$, of differential forms.  If
$u_1 \in \Lambda^j(\Omega)$ and $u_2 \in \Lambda^k(\Omega)$, then $u_1
\wedge u_2 \in \Lambda^{j+k}(\Omega)$, and the identity
\begin{equation}
  \label{tensor-to-wedge}
  u_1 \wedge u_2 = 
  \binom{k+j}{k} \skw (u_1 \otimes u_2)
\end{equation}
holds. The exterior derivative
$d = d_k : \Lambda^k(\Omega) \to \Lambda^{k+1}(\Omega)$
is given by
  \begin{equation*}
    d u_x(v_1, \ldots, v_{k+1})
    = \sum_{j=1}^{k+1} \partial_{v_j} u_x(v_1, \ldots, \hat v_j, \ldots v_{k+1}),
  \end{equation*}
where the hat is used to indicate a supressed argument.  If $F: \M \to
\M^\prime$ is a smooth map between manifolds, then the corresponding
pullback $F^*$ is a map from $\Lambda^k(\M^\prime) \to \Lambda^k(\M)$
given by
  \begin{equation*}
    (F^*u)_x (v_1, \ldots, v_k) = u_{F(x)}(DF_x(v_1), \ldots, DF_x(v_k)).
  \end{equation*}
In most of the applications in this paper, $\M$ will be domain in
$\R^n$, $\M^\prime$ a domain in $\R^{m+1}$, and $DF_x$ the Jacobian
matrix of the map $F$.  When $\M$ is a submanifold of $\M^\prime$,
then the pullback by the inclusion map, $\Lambda^k(\M^\prime) \to
\Lambda^k(\M)$, is the trace map $\tr_{\M}$.  The notation
$H\Lambda^k(\Omega)$ refers to the Sobolev space
\begin{equation*}
H \Lambda^k(\Omega) = \{u \in L^2 \Lambda^k(\Omega) :
d u \in L^2 \Lambda^{k+1}(\Omega)\},
\end{equation*}
where $u \in L^2 \Lambda^k(\Omega)$, if for all vectors $v_1, \ldots ,
v_k \in \R^n$, the function $u_x(v_1, \ldots, v_k) \in L^2(\Omega)$ as
a function of $x$.  In three dimensions, the spaces $H
\Lambda^k(\Omega)$, $k=0, \ldots,3$ are the spaces $H^1(\Omega),
H(\curl,\Omega), H(\div,\Omega), L^2(\Omega)$.  The corresponding
space of piecewise smooth $k$-forms with single valued traces with
respect to $\T$, $\Lambda^k(\T)$, will then be a subspace of
$H\Lambda^k(\Omega)$. Furthermore, the set of piecewise polynomial
$k$-forms of order $r$, $\P_r\Lambda^k(\T)$, and the corresponding
space of trimmed piecewise polynomials, $\P_r^-\Lambda^k(\T)$, will
satisfy
\[
\P_{r-1}\Lambda^k(\T) \subset \P_r^-\Lambda^k(\T)
\subset \P_r\Lambda^k(\T) \subset \Lambda^k(\T) 
\subset H\Lambda^k(\Omega).
\]
In three dimensions, the spaces $\P_r^-\Lambda^k(\T)$ are the
generalizations by N\'ed\'elec \cite{nedelec-1980} of the
two-dimensional Raviart-Thomas elements and the spaces
$\P_r\Lambda^k(\T)$ are the generalizations by N\'ed\'elec
\cite{nedelec-1986} of the two-dimensionalal Brezzi-Douglas-Marini
elements.

If $f \in \Delta_m(\T)$, then $f$ corresponds to an ordered subset of
the vertices, $\Delta_0(\T)$.  We assume that all the vertices are
numbered by a set of integers $\I = \{ 0, 1, \ldots ,N(\T) \}$ such
that
\[
\Delta_0(\T) = \{ x_i \, : \, i \in \I \, \}.
\]
Any $f \in \Delta_m(\T)$ is of the form $ f= [x_{j_0}, x_{j_1}, \ldots
  x_{j_m}]$, where $j_0 < j_1 < \ldots < j_m$ and $I(f) := \{j_0,j_1,
\ldots, j_m\} \subset \I$.  Here we have used the notation $[ \cdot,
  \cdots ,\cdot ] $ to denote convex combinations.  Furthermore, the
statement $g \in \Delta(f)$ means that $g$ is a subsimplex of $f$, and
with increasingly ordered vertices.  We define the map $\sigma_f :
\Delta_0(f) \to \{0, \ldots ,m\}$ by
\[
\sigma_f(x_{j_i}) = i .
\]
In other words, $\sigma_f(y)$ gives the internal numbering of a vertex
$y$ relative to the simplex $f$.  The number of vertices in $f$ is
denoted $|f|$, i.e., $|f| = m+1$ if $f \in \Delta_m(\T)$.  If $e,f \in
\Delta(\T)$, with a disjoint set of vertices, and such that the union
of the vertices defines a simplex, then although $e$ and $f$ are
increasingly ordered, the simplex $[e,f]$ will not necessarily be
increasingly ordered.  We then denote by $\<e,f\>$ the increasingly
ordered simplex composed of the vertices of $e$ and $f$.  The set
$\bar \Delta(f)$ contains the emptyset, $\emptyset$, in addition to
the elements of $\Delta(f)$, and $\emptyset $ is the single element of
$\Delta_{-1}(\T)$.  Since the ordering of a simplex $ f= [x_{j_0},
  x_{j_1}, \ldots x_{j_m}] \in \Delta_m(\T)$ is inherited from the
global numbering of the vertices, the various simplices are not
necessarily equally oriented.  If $m=n$, we define the orientation of
$f$, $o(f)$, by
\begin{equation*}
o(f) = \sign \det \Big(x_{j_1}- x_{j_0}, \ldots , x_{j_m}- x_{j_0} \Big),
\end{equation*}
i.e., it is the sign of the determinant of a nonsingular $n \times n$
matrix. 

In the analysis that follows, we will make use of some concepts from
simplicial complexes, e.g., see \cite{spanier}.  The $k$-chains
defined by the mesh $\T$ is a vector space consisting of linear
combinations of the form $\sum_{f \in \Delta_k} c_f f$. We will let
$\C_k$ denote the corresponding space of vector representations, i.e.,
if $c \in \C_k$ then $c = \{c_f\}_{f \in \Delta_k}$, $c_f \in \R$. In
fact, in the development of the theory below, we will consider
$k$-chains with values in a finite dimensional vector spaces $X$, by
which we mean spaces of the form $\C_k\otimes X$. In particular, $X$
will be a subspace of $\P_1^-\Lambda^k(\T)$. The {\it boundary
  operator} is a chain map,
  \[
  f \mapsto \sum_{i \in I(f)} (-1)^{\sigma_f(x_i)} f(\hat x_i),
  \]
where the hat is used to indicate a suppressed argument. Using vector
representations, the corresponding operator $\partial_k : \C_k \otimes
X \to \C_{k-1}\otimes X$ can be expressed as
  \[
(\partial_k c)_{f} = \sum_{i \in \I} c_{[x_i,f]} \equiv \sum_{i \in \I} 
  (-1)^{\sigma_{\<x_i,f\>}(x_i)} c_{\<x_i,f\>}, \quad f \in \Delta_{k-1}(\T),\,
  1 \le k \le n,
\]
where $c_{\<x_i,e\>} = 0$ if $\<x_i,e\>$ is not an element of
$\Delta_k(\T)$.  We can also identify $\C_{-1}\otimes X$ as $X$,
corresponding to the single element $\emptyset$ of $\Delta_{-1}(\T)$,
and $\partial_0 : \C_0 \otimes X \to X$ by $\partial_0 c = \sum_{i \in
  \I} c_{x_i}$.

The corresponding {\it coboundary operators} are cochain maps
$\delta_k : \C_k \otimes X \to \C_{k+1}\otimes X$ given by
\[
(\delta_{k} c)_f = \sum_{i \in I(f)}
  (-1)^{\sigma_f(x_i)} c_{f(\hat x_i)},
\quad f \in \Delta_{k+1}(\T), 
\]
for $-1 \le k \le n-1$.  If $X$ and $Y$ are finite dimensional vector
spaces, then these definitions lead to the identity
\begin{equation}
  \label{sum-by-parts}
  \sum_{f \in \Delta_{k-1}(\T)} (\partial_k c)_f \otimes \tilde c_f
  = \sum_{ f \in \Delta_{k}(\T)} c_f \otimes  (\delta_{k-1}\tilde c)_f,
  \quad c \in \C_{k} \otimes X,\, \tilde c \in \C_{k-1} \otimes Y,
\end{equation}
i.e., $\partial$ is the adjoint of $\delta$ with respect to the inner
product on $\C_k$. Furthermore, the operators $\partial$ and $\delta$
will satisfy the complex properties $\partial^2 = \delta^2 = 0$.

Associated to any vertex $x_i \in \Delta_0(\T)$, we denote by
$\lambda_i(x)$ the nonnegative piecewise linear function equal to one
at $x_i$ and zero at all other vertices.  More generally, if $f $ is a
ordered subset of $\Delta_0(\T)$, but not necessarily increasingly
ordered, then $\phi_f$ will denote the Whitney form associated to $f$.
More precisely, if $f = [x_{j_0}, x_{j_1}, \ldots, x_{j_m}]$, then
$\phi_f$ is given by
\[
\phi_f = m!\sum_{i =0}^m (-1)^{i} \lambda_{j_i} d\lambda_{j_0} \wedge 
\ldots \wedge \widehat{d\lambda_{j_i} } \wedge
\ldots \wedge d\lambda_{j_m}.
\]
In particular, for $f = [x_i]$, $\phi_f = \lambda_i$.  If $m>0$, then
$\int_{g} \phi_f$ is plus or minus one if $g = f$, and zero if $g \in
\Delta_m(\T)$, $g \neq f$.  The forms $\{ \phi_f \}_{f \in
  \Delta_m(\T)}$ span the space $\P_1^-\Lambda^m(\T)$ and they are
local with support in $\Omega_f$.  It can also be easily checked that
\begin{equation}
  \label{d-phi}
  d\phi_f = (m+1)! d\lambda_{j_0} \wedge \ldots \wedge d\lambda_{j_m}
  = \sum_{j \in \I}\phi_{[x_j,f]} = (\partial_{m+1} \phi)_f,
\end{equation}
where $\phi_{[x_j,f]} = 0$ if $[x_j,f]$ does not correspond to a
subsimplex of $\T$.  Here the operator $\partial_{m+1}$ has the
interpretation given above, where $\{ \phi_f \}$ for $f \in
\Delta_{m+1}(\T)$ should be seen as an element of $\C_{m+1} \otimes
\P_1\Lambda^{m+1}(\T)$.  Furthermore, if $u \in \P_1^-\Lambda^m(\T)$
is expanded, such that $ u = \sum_{f \in \Delta_m(\T)} c_f \phi_f$,
then
\begin{equation}
  \label{span-d}
  du =  \sum_{f \in \Delta_m(\T)} c_f (\partial_{m+1} \phi)_f
  =  \sum_{f \in \Delta_{m+1}(\T)} (\delta_m c)_f \phi_f.
\end{equation}
It is also a consequence of the definition of $\phi_f$ and
\eqref{d-phi} that
\begin{equation}\label{phi-xf}
  \phi_{[x_i,f]} = \Big(\lambda_i d - m d\lambda_i \wedge\Big)\phi_f,
  \quad f \in \Delta_{m-1}(\T).
\end{equation}
If $0 \le m \le n$, then associated to the simplex $f = [x_{j_0},
  x_{j_1}, \ldots, x_{j_m}]$, we define the standard simplex $\S_f$ by
\[
\S_f = \big\{\, \lambda =  (\lambda_{j_0}, \lambda_{j_1}, \ldots, \lambda_{j_m})
  \in \R^{m+1} \, : \,
\sum_{i =0}^{m} \lambda_{j_i} = 1, \quad \lambda_{j_i} \ge 0 \,  \big\},
\]
while $\S_f^c$  is the set of all convex combinations
between $\S_f$ and the origin, i.e., $\S_f^c = [0,\S_f]$. Alternatively,
\[
\S_f^c = \big\{\, \lambda =  (\lambda_{j_0}, \lambda_{j_1}, \ldots, \lambda_{j_m})
\in \R^{m+1} \, : \,
\sum_{i =0}^{m} \lambda_{j_i} \le  1, \quad \lambda_{j_i} \ge 0 \,  \big\}.
\]
If $f= [x_{j_0}, x_{j_1}, \ldots x_{j_m}] \in \Delta_m(\T)$, then $L_f
: \Omega \to \S_f^c$ will be the map
\[
L_f(x) = \big(\lambda_{j_0}(x), \lambda_{j_1}(x), \ldots , \lambda_{j_m}(x)\big).
\]
If $f \in \Delta_0(\T)$, $\S_f^c = [0,1]$, while in the special case
$f = \emptyset$, we let $\S_{\emptyset} = \S_{\emptyset}^c = \{0\}$
and $L_f(\Omega)= 0$.

\begin{figure}[htb]
\setlength{\unitlength}{0.35cm}
\centering
\begin{picture}(30,15)
\linethickness{2pt}
\put(5,0){\line(0,1){10}}
\linethickness{1pt}
\put(5,0){\line(2,1){10}}
\put(5,10){\line(2,-1){10}}
\put(5,0){\line(-3,4){7.5}}
\put(5,10){\line(-1,-1){4.4}}
\put(5,10){\line(-1,0){7.5}}
\put(20,0){\vector(0,1){10}}
\put(20,0){\vector(1,0){10}}
\put(22.4,3.7){\line(1,-1){1}}
\put(30.5,0){$\lambda_0$}
\put(19.8,10.5){$\lambda_1$}
\linethickness{2pt}
\put(20,10){\line(1,-1){10}}
\linethickness{1pt}
\put(6.7,3.5){$x$}
\put(5.1,5){$f$}
\put(4.8,-.6){$x_0$}
\put(4.8,10.5){$x_1$}
\put(7.5,5.5){$\Omega_f$}
\put(2.5,5.5){$\Omega_f$}
\put(20.5,2){$(\lambda_0(x), \lambda_1(x))$}
\put(21,8){$(0,\lambda_1(y))$}
\put(25.5,4.5){$\S_f$}
\put(21.5,4.5){$\S_f^c$}
\put(15,7){$L_f$}
\put(6.7,2.7){$\bullet$}
\put(22.5,2.8){$\bullet$}
\put(6.7,2.7){$\bullet$}
\put(22.5,2.8){$\bullet$}
\put(0.5,7.4){$\bullet$}
\put(19.7,7.9){$\bullet$}
\put(0.5,8.5){$y$}
\qbezier(7,3)(15,9)(23,3)
\qbezier(0.6,7.6)(15,12)(20,8)
\end{picture}
\vskip.1in
\caption{The map $x \mapsto L_f(x)$ and $y \mapsto L_f(y)$ for $n=2$,
  $m=1$, and $f=[x_0,x_1]$.}
\label{fig:map}
\end{figure}

In the construction below, we will frequently use the pullback $L_f^*$
mapping $\Lambda^k(\S_f^c)$ to $\Lambda^k(\T)$. For $m=1$, $f =
[x_0,x_1]$, the map $L_f(x) = (\lambda_0(x), \lambda_1(x))$, and the
one-form $w_{\lambda} = w_0(\lambda_0,\lambda_1) d \lambda_0 +
w_1(\lambda_0,\lambda_1) d \lambda_1$ defined on $\S_f^c$,
  \begin{multline*}
      (L_f^* w)_x(v) = w(L_f(x))(DL_f(v))
      \\
  = w_0(\lambda_0(x),\lambda_1(x)) (\nabla \lambda_0(x) \cdot v)
    +  w_1(\lambda_0(x),\lambda_1(x)) (\nabla \lambda_1(x) \cdot v).
  \end{multline*}
The map $L_f^*$ takes smooth forms on $\S_f^c$ to piecewise smooth
forms on $\Omega$, and also polynomial forms to piecewise polynomial
forms.  More precisely, if $g \in \Delta(f)$, then
\begin{equation}\label{L-star-prop}
L_g^* \Big(\P\Lambda^k(\S_f^c)\Big) \subset \P\Lambda^k(\T),
\end{equation}
where $\P$ is either $\P_r$ or $\P_r^-$.  Another key tool we will
utilize below is the piecewise linear function $\rho_f$, defined by
\[
\rho_f(x) = 1 - \sum_{i \in I(f)} \lambda_i(x),
\]
which can be seen as a distance function between $f \in \bar
\Delta(\T)$ and $x \in \Omega$.  Note that $0 \le \rho_f(x) \le 1$ and
$\rho_f \equiv 1$ if $f = \emptyset$. Alternatively, we have $\rho_f =
L_f^* b$, where $b = b_f : \S_f^c \to \R$ is the distance to the
origin, i.e., $b(\lambda) = 1 - \sum_{i =0}^{m} \lambda_{j_i}$.

\subsection{The macroelements}
\label{sec:macroelements}
We recall that for any $f \in \Delta(\T)$, we denote by $\Omega_f$ the
union of the simplices in $\Delta_n(\T)$ containing $f$ as a
subsimplex, while $\T_f$ is the restriction of the mesh $\T$ to
$\Omega_f$.  We refer to $\Omega_f$ as the macroelement of $f$, or
alternatively, as the star of $f$. The domain $\Omega_f$ is
contractible with respect to any $x \in f$.  If $f \in
\Delta_{m-1}(\T)$ and $T \in \Delta_n(\T_f)$, we let $f^*(T) \in
\Delta_{n-m}(T)$ be the subsimplex of $T$ opposite $f$.  The link of
$f$, denoted by $f^*$, is given as $f^* = \bigcup_{T \in
  \Delta_n(\T_f)} f^*(T)$ (e.g., see \cite{daverman-sher}).  More
precisely,
\begin{equation*}
  f^* = \{x \in \Omega_f: \lambda_i(x) =0, i \in I(f)\}.
  \end{equation*}
The link $f^*$ can be viewed as an $n-m$ dimensional oriented manifold
composed of the simplices $f^*(T)$, where $f^*(T)$ has an orientation,
$o(f^*(T),T)$, induced by the $n$-simplex $T$. More precisely, if $f =
[x_{j_0},\ldots , x_{j_{m-1}}]$, then
\begin{equation*}
o(f^*(T),T) = o(T) \prod_{i = 0}^{m-1} (-1)^{\sigma_{T_i}(x_{j_i})},
\end{equation*}
where $T_i = T(\hat x_{j_0}, \ldots , \hat x_{j_{i-1}}) \in
\Delta_{n-i}(T)$.  Any $x \in \Omega_f$ can be written uniquely as a
convex combination of the points $x_i, \, i \in I(f)$, and points in
$f^*$, i.e., $\Omega_f = [f, f^*]$.  In the special case when $m=n-1$,
the manifold $f^*$ will be reduced to two vertices, or only one close
to the boundary, while in the special case $f = \emptyset$, we have
$\Omega_f = f^* = \Omega$.  Below we will also encounter the extended
macroelement, $\Omega_f^E$, defined by $\Omega_f^E = \cup_{i \in I(f)}
\Omega_{x_i}$.

\vskip.25in
\setlength{\unitlength}{0.30cm}
\begin{figure}[ht]
 \vspace*{.25in}\begin{center}
\begin{picture}(30,0)
\put(0,0){\line(1,0){15}}
\put(0,0){\line(1,1){5}}
\put(0,0){\line(1,-1){5}}
\put(5,5){\line(1,0){5}}
\put(5,-5){\line(1,0){5}}
\put(15,0){\line(-1,1){5}}
\put(15,0){\line(-1,-1){5}}
\put(7.5,0){\line(-1,2){2.5}}
\put(7.5,0){\line(1,2){2.5}}
\put(7.5,0){\line(-1,-2){2.5}}
\put(7.5,0){\line(1,-2){2.5}}
\put(20,0){\line(1,-1){5}}
\put(20,0){\line(1,1){5}}
\put(20,0){\line(1,0){10}}
\put(30,0){\line(-1,1){5}}
\put(30,0){\line(-1,-1){5}}
\put(25,-0.8){$f$}
\put(24.8,5.2){$f^*$}
\put(25,-5.5){$f^*$}
\put(7.2,-1){$f$}
\put(2.4,-4){$f^*$}
\put(2.4,3.5){$f^*$}
\put(12.2,-3.7){$f^*$}
\put(12.4,3.4){$f^*$}
\put(7.5,5.2){$f^*$}
\put(7.5,-5.7){$f^*$}
\end{picture}
\vskip.8in
\caption{\qquad \qquad  2D Vertex macroelement  \qquad \qquad 
\qquad \qquad \qquad 2D Edge macroelement
\newline
When $f$ is a vertex, $f^*$ is the boundary of the vertex macroelement.
When $f$ is an edge, $f^*$ is the two remaining vertices of the macroelement.}
\label{fig:vertex-edge-macroelement}
\end{center}
\end{figure}

\begin{center}
\begin{figure}[ht]
\hspace*{8em}
\begin{tikzpicture}[scale = .3]
\draw [dashed] (-4,2.5) -- (4, 0) -- (8,5.6);
\draw [ ultra thick](4,8.5) -- (-4,2.5);
\draw [ ultra thick](4,8.5) -- (8,5.5);
\draw [ ultra thick] (4,8.5) -- (0,-5.6);
\draw [ ultra thick] (4,-8.5) -- (0,-5.6);
\draw [ ultra thick] (4,-8.5) -- (12,-2);
\draw [ ultra thick] (4,8.5) -- (12,-2);
\draw [dashed]  (-4,2.5) -- (8,5.5) ;
\draw [dashed] (4,0) -- (4,8.5);
\draw [dashed] (4,0) -- (4,-8.5);
\draw [dashed] (4,0) -- (0,-5.6);
\draw [ultra thick] (8,5.5) -- (12,-2);
\draw [dashed] (4,0) -- (12,-2);
\draw [ultra thick] (-4,2.5) -- (0,-5.6);
\draw  [ultra thick] (0,-5.6) -- (12,-2);
\node [right] at (8,5.6) {$x_2$};
\node [left] at (-4,2.5) {$x_3$};
\node [left] at (0,-5.6) {$x_4$};
\node [right] at (12,-2) {$x_5$};
\node[below] at (4.2,0) {$x_0$};
\node [above] at (3.4,8.6) {$x_1$};
\node [below] at (3.4,-8.6) {$x_6$};
\end{tikzpicture}
\caption{The vertex macroelement $\Omega_f \subset \R^3$, where $f = [x_0]$
  and $f^*$ is the union of the faces containing three of the other vertices
  (excluding $x_0$)}
\label{fig:3d-vertex-macroelement}
\end{figure}
\end{center}

\begin{center}
\begin{figure}[!ht]
\vspace*{.25in}
\hspace*{8em}
\begin{tikzpicture}[scale = .3]
\draw [dashed] (-4,2.5) -- (4, 0) -- (8,5.6);
\draw (4,8.5) -- (-4,2.5);
\draw (4,8.5) -- (8,5.5);
\draw (4,8.5) -- (0,-5.6);
\draw (4,8.5) -- (12,-2);
\draw [dashed] [ ultra thick] (-4,2.5) -- (8,5.5) ;
\draw [ultra thick] (4,0) -- (4,8.5);
\draw [dashed] (4,0) -- (0,-5.6);
\draw [ultra thick] (8,5.5) -- (12,-2);
\draw [dashed] (4,0) -- (12,-2);
\draw [ultra thick] (-4,2.5) -- (0,-5.6);
\draw [ultra thick] (0,-5.6) -- (12,-2);
\node [right] at (8,5.6) {$x_2$};
\node [left] at (-4,2.5) {$x_3$};
\node [left] at (0,-5.6) {$x_4$};
\node [right] at (12,-2) {$x_5$};
\node[below] at (4.2,0) {$x_0$};
\node [above] at (3.4,8.6) {$x_1$};
\end{tikzpicture}
\caption{The edge macroelement $\Omega_f \subset \R^3$, where $f =
  [x_0,x_1]$ and $f^*$ is the closed curve connecting the vertices
  $x_2,x_3,x_4$, and $x_5$.}
\label{fig:3d-edge-macroelement}
\end{figure}
\end{center}

\begin{center}
\begin{figure}[!ht]
\vspace*{-.25in}
\hspace*{8em}
\begin{tikzpicture}[scale = .3]
\draw (4,8.5) -- (-4,2.5);
\draw  [ultra thick] (4,8.5) -- (8,5.5);
\draw [ultra thick] (4,8.5) -- (0,-5.6);
\draw (4,8.5) -- (12,-2);
\draw [dashed]  (-4,2.5) -- (8,5.5) ;
\draw [ultra thick] [dashed] (8,5.6) -- (0,-5.6);
\draw  (8,5.5) -- (12,-2);
\draw  (-4,2.5) -- (0,-5.6);
\draw  (0,-5.6) -- (12,-2);
\node [right] at (8,5.6) {$x_2$};
\node [left] at (-4,2.5) {$x_3$};
\node [left] at (0,-5.6) {$x_4$};
\node [right] at (12,-2) {$x_5$};
\node [above] at (3.4,8.6) {$x_1$};
\end{tikzpicture}
\caption{The face macroelement $\Omega_f \subset \R^3$, where
  $f = [x_4,,x_2,x_1]$ and $f^*$ is the vertices $x_3$, and $x_5$.}
\label{fig:3d-face-macroelement}
\end{figure}
\end{center}

We define $\C_k(f^*)$ as the space of vector representations of
$k$-chains defined on the manifold $f^*$, i.e., if $c \in \C_k(f^*)$
then $c = \{c_f\}_{f \in \Delta_k(f^*)}$.  The corresponding boundary
and coboundary operators, $\partial_k(f^*)$ and $\delta_k(f^*)$, are
defined as the operators $\partial_k$ and $\delta_k$ above, but by
restricting to the manifold $f^*$.  In the construction performed
later in this paper, we will utilize the fact that for any $f \in
\Delta_{m-1}(\T)$, $ 1 \le m \le n-1$, the cochain complex
\begin{equation}\label{f*-complex}
\begin{CD}
  \R  @>>> \C_0(f^*)  @>\delta>>   \C_1(f^*) @>\delta>>
  \cdots @>\delta>>  \C_{n-m}(f^*)   @>\delta>>  \R
\end{CD}
\end{equation}
is exact, where $\delta = \delta(f^*)$, and where the special operator
$\delta_{n-m}(f^*) :\ C_{n-m}(f^*) \to \R$ will be defined below.
This exactness is a consequence of the corresponding property for the
trimmed linear forms restricted to $f^*$.  Since $f^*$ is a piecewise
flat submanifold of the boundary of $\Omega_f$, defined by the mesh
$\T$, we can consider the trace spaces $\P_1^- \Lambda^k(f^*)$,
defined as $\tr_{f^*}\P_1^- \Lambda^k(\T)$.  If $f$ is an interior
simplex, such that $f^*$ is a manifold without a boundary, the complex
\begin{equation}\label{trimmed-complex}
\begin{CD}
\R  @>>> \P_1^- \Lambda^0(f^*)  @>d>>   \P_1^- \Lambda^1(f^*)   @>d>>  \cdots 
\P_1^- \Lambda^{n-m}(f^*)    @> >>  \R
\end{CD} 
\end{equation}
is exact, where the final arrow represents the integral over $f^*$.
If $f$ is a boundary simplex, the manifold $f^*$ may have a boundary,
and in this case the last arrow in \eqref{trimmed-complex} is
redundant.  By expanding the elements of $\P_1^- \Lambda^{k}(f^*)$ in
the basis functions, and using the identity \eqref{span-d}, we obtain
the equivalent complex \eqref{f*-complex}.  More specifically, the map
$l_k: \P_1^- \Lambda^{k}(f^*) \to \C_k(f^*)$ given by $\sum_{e \in
  \Delta_k(f^*)} c_e \phi_e \mapsto c_e$ gives the commuting relation
$\delta_k(f^*) \circ l_k = l_{k+1} \circ d$. As a consequence, the
exactness of the complex \eqref{f*-complex} follows from the exactness
of \eqref{trimmed-complex}.  The special operator $\delta_{n-m}(f^*)$
is defined by
\begin{equation}
\label{spec-op}
\delta_{n-m}(f^*)c = \sum_{e \in \Delta_{n-m}(f^*)} o(e,T_e) c_{e},
\quad T_e = \<e,f\>.
\end{equation}
Note that this operator is always well-defined, but that the last
arrow in \eqref{f*-complex} may be redundant if $f$ intersects the
boundary of $\Omega$. Since $\partial$ is the adjoint of $\delta$ with
respect to the inner product of $\C(f^*)$, the following result is a
consequence of the exactness of the complex \eqref{f*-complex}.

\begin{lem}\label{lem:exact-f*}
Let $f \in \Delta_{m-1}(\T)$ and $0 \le k < n-m$. If $c \in \C_k(f^*)$
satisfies $\partial_k(f^*) c = 0$, then there is a $\tilde c \in
\C_{k+1}(f^*)$ such that $c = \partial_{k+1}(f^*) \tilde
c$. Furthermore, $\tilde c$ is uniquely determined if we require
$\delta_{k+1}(f^*) \tilde c= 0$.
\end{lem}

Below we will use  Lemma~\ref{lem:exact-f*} in a
generalized sense, where we consider
$\C_k$ with values in a finite dimensional vector space $X$,
cf. Sections~\ref{sec:scalar} and \ref{sec:mu-functions}.

\subsection{The average operators and their generalizations}
\label{sec:average}
A key tool for our construction below is a family of average
operators, $A_f^k$, where $f \in \Delta$, which
map piecewise smooth $k$-forms on $\Omega_f$ to smooth $k$-forms on
$\S_f^c$. The operators $A_f^k$
will be defined by a function
$G_f : \Omega_f \times \S_f^c \to \Omega_f$,
given by
\[
G_f(y,\lambda) = \sum_{i \in \I(f)} \lambda_i x_i + b(\lambda)y.
\]
Note that if $x \in f$ then, since $b(L_fx) = 0$, we have
\begin{equation}\label{G-L_f-rel}
G_f(y, L_f x) = x, \quad x \in f.
\end{equation}
For each fixed $y \in \Omega_f$, $G_f(y, \lambda)$ is linear with
respect to $\lambda$.  The corresponding derivative with respect
$\lambda$, $DG_f(y, \cdot)$, is therefore an operator mapping tangent
vectors of $\S_f^c$, $TS_f^c = \R^{m+1}$, into $T\Omega_f = \R^n$ which is
independent of $\lambda$.  It is given by
\[
DG_f(y, \cdot) = \sum_{i \in I(f)} (x_i - y) d\lambda_i.
\]
Since for each $y$, the map $G_f(y, \cdot)$ maps $\S_f^c$ to
$\Omega_f$, the corresponding pullback, $G_f(y, \cdot)^*$, maps
$\Lambda^k(\Omega_f)$ to $\Lambda^k(\S_f^c)$.  As a further
consequence, an average of these maps over $\Omega_f$ with respect to
$y$ will also map $\Lambda^k(\Omega_f)$ to $\Lambda^k(\S_f^c)$.  In
order to define the averages we want, we will introduce a family of
piecewise constant $n$-forms, $z_f \in \P_1^-\Lambda^n(\T_f)$, with
the property that $z_f$ has support in $\Omega_f$ and
\begin{equation}\label{average-prop}
\int_{\Omega_f} z_f = \int_{\Omega} z_f =1. 
\end{equation}
The operator $A_f^k$ is then defined for $f \in \Delta_m$ by
\[
A_f^k u  
= \int_{\Omega} G_f(y, \cdot)^*u\wedge z_f.
\] 
Since pullbacks commute with the exterior derivative, so do the
operators $A_f^k$, i.e., $d A_f^k = A_f^{k+1} d$, and from
\eqref{G-L_f-rel}, we obtain that for $f \in \Delta_m$,
\begin{equation*}
\tr_{f} L_f^*A_f^k  = \tr_f, \quad m\ge k.
\end{equation*}
Of course, the properties of the operators $A_f^k$ will also depend on
the choice of the functions $z_f$.  For $f \in \Delta_n$, these
functions are defined to be
\begin{equation}\label{def-zf-n}
z_f = \frac{\kappa_f}{|\Omega_f|} \vol,
\end{equation}
where $\kappa_f$ is the characteristic function of $\Omega_f$, while
for $ f \in \Delta_m, \, 0 \le m <n$, the functions $z_f$ are defined
recursively by the relation
\begin{equation}\label{def-zf}
  z_f= \frac{1}{|f^*|} \sum_{i \in I(f^*)}
  z_{\<x_i,f\>}.
\end{equation}

We note that it follows by construction that all the functions $z_f$
have support in $\Omega_f$ and satisfy relation \eqref{average-prop}.
Furthermore, it is a consequence of Lemma 2.1 of \cite{bubble-II} that
the operators $A_f^k$ map $\Lambda^k(\T_f)$ to $\Lambda^k(S_f^c)$ and
also map piecewise polynomial forms to polynomial forms. More
precisely, we have for $f \in \Delta$,
\begin{equation}\label{A-prop}
A_f^k (\P_r\Lambda^k(\T_f) ) \subset \P_r\Lambda^k(\S_f^c), \qquad
A_f^k (\P_r^-\Lambda^k(\T_f) ) \subset \P_r^-\Lambda^k(\S_f^c).
\end{equation}

\begin{remark} In the argument given in \cite{bubble-II}, it was assumed  that 
the functions $z_f$ were given by \eqref{def-zf-n} for all
$f$. However, the modification we need to cover the more general
average functions $z_f$ introduced above is straightforward.
\end{remark}

Consider the case of scalar valued functions, i.e., the case $k=0$.
Since all the functions $z_f$ satisfy the identity
\eqref{average-prop}, it follows that for $f \in \Delta$, and $e \in
\Delta_1(f)$,
\[
(\delta A^0 u)_{e,f} :=  \sum_{i \in I(e)} (-1)^{\sigma_e(x_i)} A_{f(\hat x_i)}^0 u
\]
will be zero when $u$ is a constant. Therefore, this expression only
depends on $du$.  Below we will construct a corresponding operator
$R^1 $, mapping one-forms to zero-forms, such that
\begin{equation}\label{RA-delta}
R^1du = (\delta A^0 u)_{e,f}, \quad e \in \Delta_1(f).
\end{equation}
A natural choice seems to be to label this operator by the simplex
pair $(e,f)$.  In fact, this was the choice used in
\cite{bubble-II}. However, for the theory developed in this paper, it
appears more appropriate to use the equivalent label, $(e, f\cap
e^*)$, where $f \cap e^*$ belongs to $e^*$ and vice versa.

More generally, we introduce the sets $\Delta_{j,m} =
\Delta_{j,m}(\T)$ of pairs of simplices, given by
\[
\Delta_{j,m} := \{ (e,f)  \, : \,  f \in \Delta_{m}(\T),
\, e \in \Delta_j(f^*)\, \}
\]
for $-1 \le m \le n$ and $0 \le j < n-m$. For $(e,f) \in
\Delta_{j,m}$, $j \ge 0$, we will define operators $R_{e,f}^k $
mapping $k$-forms to ($k-j$)-forms.  We will refer to these operators
as order reduction operators.  To define them, we recall that the map
$G_f$ maps $\Omega_f \times \S_f^c$ to $\Omega_f$, and as a
consequence, the corresponding pullback, $G_f^*$, is a map
\[
G_f^* : \Lambda^k(\Omega_f) \to \Lambda^k(\Omega_f \times \S_f^c)
\]
and we can express
\[
\Lambda^k(\Omega_f \times \S_f^c) 
= \sum_{j=0}^k \Lambda^j(\Omega_f) \otimes \Lambda^{k-j}(\S_f^c).
\]
Furthermore, for each $0 \le j \le k$, there is a canonical map $\Pi_j
: \Lambda^k(\Omega_f \times \S_f^c) \to \Lambda^j(\Omega_f) \otimes
\Lambda^{k-j}(\S_f^c)$ such that $\sum_{j=0}^k \Pi_j $ is the
identity.  We refer to \cite[Section 5.1]{bubble-II} for more details.
As in \cite{bubble-II}, all the operators $R_{e,f}^k$ will be of the
form
\begin{equation}\label{form-R}
  (R_{e,f}^k u)_{\lambda} =  \int_{\Omega} (\Pi_j G_f^*u)_{\lambda} \wedge z_{e,f},
  \quad  \lambda \in  \S_f^c,
\end{equation}
where the functions $z_{e,f}$ are trimmed linear $n-j$ forms for
$(e,f) \in \Delta_{j,m}$, and with support in $\Omega_{f} \cap
\Omega_e^E$.  The construction of the functions $\{z_{e,f}\}$, given
in Section~\ref{sec:mesh-functions} below, will deviate from the
corresponding construction given in \cite{bubble-II}.  In fact, the
careful construction of these functions below represents the main tool
for obtaining the improved results we derive in this paper, as
compared to the results presented in \cite{bubble-II}.

For pairs $(e,f) \in \Delta_{0,m}$, i.e., when $e$ is a vertex, we
define $z_{e,f} =-z_{\<e,f\>}$, where the $n$-forms $z_{\<e,f\>}$ are
defined by \eqref{def-zf-n} and \eqref{def-zf}.  As a consequence, for
any $e \in \Delta_0(f^*)$, we have $R_{e,f}^k = -A_{\<e,f\>}^k$.  The
functions $\{z_{e,f}\}$ will be constructed to satisfy the relation
\begin{equation}\label{z-prop}
d z_{e,f} = (-1)^{j+1}(\delta z)_{e,f}, \quad (e,f) \in \Delta_{j,m},
\end{equation}
for $j \ge 1$, where the generalized coboundary operator defined for
pairs of simplices in $\{\Delta_{j,m} \}$ is given by
\[
(\delta z)_{e,f} = \sum_{i \in I(e)} (-1)^{\sigma_e(x_i)} z_{e(\hat x_i),f}.
\]
\begin{remark} The identity, $R_{e,f}^k = -A_{\<e,f\>}^k$ and the definition of 
the $\delta$ operator appear to deviate from the corresponding
relations in \cite{bubble-II}.  However, as stated above, in the
present paper it is more convenient to use a different, but
equivalent, labeling of the $z$ functions and $R$ operators as
compared to the previous paper. More precisely, the operators
$R_{e,f}^k$ introduced above were labeled by the pair $(e,\<e,f\>)$ in
\cite{bubble-II}, and this causes minor differences in the relations
above, and also at a few occurrences below.
\end{remark}
As a consequence of the definition of the operators $R_{e,f}^k$, given
by \eqref{form-R} and \eqref{z-prop}, the operators $R_{e,f}^k$ will
satisfy the relation
 \begin{equation}\label{Rkd-delta}
 R_{e,f}^{k+1} du = (-1)^j dR_{e,f}^k u - (\delta R^ku)_{e,f}, 
\quad (e,f) \in \Delta_{j,m}, \, 0 \le j \le k+1,
 \end{equation}
where $\Delta_{j,m}$ is defined to be the emptyset if $j \ge
n-m$. Furthermore, $(\delta R^ku)_{e,f}$ is taken to be zero for $e
\in \Delta_0$.  When $e \in \Delta_{k+1}(f^*)$, $R_{e,f}^k u = 0$ and
\eqref{Rkd-delta} reduces to
 \begin{equation}\label{Rkd-delta-spec}
   R_{e,f}^{k+1} du =  - (\delta R^ku)_{e,f}, \quad (e,f) \in
          \Delta_{k+1,m},
 \end{equation}
which is consistent with \eqref{RA-delta}. We refer to Section 5.3 of
\cite{bubble-II} and Section~\ref{sec:order-reduct} below for more
details.

The following lemma generalizes the mapping properties \eqref{A-prop}
of the average operators $A_f^k$ to similar results for the order
reduction operators. In fact, this result corresponds to Proposition
5.2 of \cite{bubble-II}.

\begin{lem}\label{lem:R-pol-prop}
Assume that $(e,f )\in \Delta_{j,m}(\T)$.
\begin{itemize}
\item[i) ]If $u \in \Lambda^k( \T_f)$, then $b^{-j}R_{e,f}^k u 
\in \Lambda^{k-j}(\S_f^c)$.
\item[ii)] If 
$u \in \P_r\Lambda^k( \T_f)$, then $b^{-j}R_{e,f}^k u 
\in \P_{r}\Lambda^{k-j}(\S_f^c)$.
\item[iii)] If $u \in \P_r^-\Lambda^k(\T_f)$, then 
$b^{-j} R_{e,f}^k u \in \P_{r}^-\Lambda^{k-j}(\S_f^c)$.
\end{itemize}
\end{lem}

\begin{remark} The proof of Proposition 5.2 of \cite{bubble-II}
carries over directly to the present setting, even if the definition
of the weight functions $z_{e,f}$ will be modified below. Therefore,
we omit the proof here.
\end{remark}

\subsection{An outline of the construction}
\label{sec:outline}
In the present notation, the trace preserving operators $C_m^k$,
introduced in \cite[Section 6]{bubble-II}, admit the representation
\begin{multline}\label{Cm-rewritten}
C_m^k u = \sum_{f \in \Delta_m} 
\sum_{g \in \bar \Delta(f)}(-1)^{|f| -|g|} L_g^* A_f^k u 
\\
+\sum_{\substack{(e,f) \in \Delta_{j,m-1} \\ 0 \le j \le n-m}}(-1)^{j-1} 
\sum_{g \in \bar \Delta(f)} (-1)^{|f|-|g|} 
\frac{\phi_e}{\rho_g} \wedge L_{g}^* b^{-j} R_{e,f}^ku.
\end{multline}
where $0 \le m \le n-1$, and where we recall that $R_{e,f}^k = 0$ if
$e \in \Delta_j$, $ j > k$. The definition of the operator $C_m^k$
contains the rational terms $\phi_e/\rho_g$. However, it is
established in \cite{bubble-II}, cf. Lemma 2.2 and Proposition 7.1 of
that paper, that the operator $C_m^k$ commutes with the exterior
derivative, preserves piecewise smoothness and the piecewise
polynomial spaces $\P_r\Lambda^k(\T)$ and $\P_r^-\Lambda^k(\T)$, and
preserves the trace of $u$ on all $f \in \Delta_m$ if $m \ge k$.

\begin{remark}
The fact that the operator $C_m^k$ preserves the trace of $u$ on all
$f \in \Delta_m$ if $m \ge k$, can be derived easily from the
definition above. In fact, by combining the first term in
\eqref{Cm-rewritten} and the second term with $j=0$, we obtain the
primal operator studied in Section 4 of \cite{bubble-II}.  This
operator can be rewritten as
\begin{equation}\label{C-primal}
C_m^k(primal) = \sum_{f \in \Delta_m} 
\sum_{g \in \bar \Delta(f)}(-1)^{|f| -|g|} \frac{\rho_f}{\rho_g}L_g^* A_f^k u.
\end{equation}
By arguing as in the proof of \cite[Lemma 4.1]{bubble-II}, using the
cancellation property with respect to $g$, we can conclude that this
operator preserves the trace on all elements of $\Delta_m$ if $m \ge
k$. Furthermore, for the second term in \eqref{Cm-rewritten} with a
fixed $f \in \Delta_{m-1}$ and $e \in \Delta_j(f^*)$, we can argue in
the same way that the resulting function has support in $\Omega_e \cap
\Omega_f = \Omega_{\<e,f\>}$. As a consequence, if $j >0$, such that
the simplex $\<e,f\> \in \Delta_s$ for $s >m$, these functions have
vanishing trace on all $m$-simplices. Hence, the trace property of
$C_m^k$ follows. The commuting property of the operator $C_m^k$ can
also be shown directly from the properties of the operators $A_f^k$
and $R_{e,f}^k$, cf. Section~\ref{sec:average} above, while the space
preserving properties require a deeper analysis, performed in
\cite[Section 7]{bubble-II}.  In fact, an alternative proof of the
space preserving properties follows as a corollary of the analysis
given in this paper.
\end{remark}

We will derive the desired decomposition \eqref{main-decomp} of
functions $u \in \Lambda^k(\T)$ from the telescoping identity
\begin{equation}\label{telescope-sum}
u = C_0^k u + \sum_{m=1}^n (C_m^k - C_{m-1}^k) u,
\end{equation}
where $C_n^k$ is the identity operator.
We have already observed,  cf. \eqref{pre-decomp}, that 
\[
(C_n^k - C_{n-1}^k) u = u - C_{n-1}^k u = \sum_{f \in \Delta_n} B_f^ku,
\]
where $B_f^k u = \tr_f (u - C_{n-1}^k u )$ is a sum of piecewise
smooth local bubbles.  For the special case $m= 0$, there are no
rational functions present in the definition of the operator $C_0^k$.
In fact, by utilizing that the range of the operator $L_{\emptyset}$
is a single point, i.e., the origin, it follows that $L_{\emptyset}^*$
maps forms of order greater than zero, to zero. As a consequence, 
we can represent the operator $C_0^k$ as
 \begin{equation}\label{C0-rewritten}
   C_0^k u = \sum_{f \in \Delta_0} \Big(L_f^*A_f^k u - L_\emptyset^* A_f^k u\Big)
   + (-1)^{k-1} \sum_{e \in \Delta_{k}}
   \phi_e \wedge L_\emptyset^* R_{e,\emptyset}^k u.
\end{equation}
In particular, from the definition of the operator
$R_{e,\emptyset}^k$, it follows that the second term in
\eqref{C0-rewritten} can be expressed as
\begin{equation*}
  W^ku :=   (-1)^{k-1} \sum_{e \in \Delta_{k}} 
   \phi_e \Big(\int_{\Omega} u \wedge z_{e,\emptyset}\Big),
\end{equation*}
which is an element of $\P_1^-\Lambda^k(\T)$.  In other words, $W^k$
is an operator which maps piecewise smooth $k$-forms into the simplest
class of piecewise polynomial $k$ forms, i.e., into trimmed linear
forms. In particular, we recall that for $e \in \Delta_k$, the
functions $z_{e,\emptyset}$ are elements of $\P_1^-\Lambda^{n-k}(\T)$,
with support in $\Omega_e^E$.  We define the local operators
$K_{0,f}^k$ by
\[
K_{0,f}^ku = L_f^*A_f^k u - L_\emptyset^* A_f^k u.
\]
The functions $K_{0,f}^k u$ will have support on $\Omega_f$, and by
using these operators, the identity \eqref{C0-rewritten} can be
written as
\[
C_0^k u = \sum_{f \in \Delta_0} K_{0,f}^k u + W^k u.
\]
As a consequence of the fact that the operators $A_f^k$ commute with
the exterior derivative, it follows that the operators $K_{0,f}^k$
will also commute with $d$.  Furthermore, the operator $W^k$ also
commutes with the exterior derivative.  In fact, from \eqref{d-phi},
we have
\[
dW^k u = (-1)^{k-1}\sum_{e \in \Delta_{k}}
(\partial  \phi)_e \Big(\int_{\Omega} u \wedge z_{e,\emptyset}\Big),
\]
and from \eqref{sum-by-parts}, \eqref{span-d}, \eqref{z-prop}, and the
Leibniz rule, we obtain

\begin{multline*}
\sum_{e \in \Delta_{k}}
(\partial  \phi)_e \Big(\int_{\Omega} u \wedge z_{e,\emptyset}\Big)
=\sum_{e \in \Delta_{k+1}}
\phi_e \Big(\int_{\Omega} u \wedge (\delta z)_{e,\emptyset}\Big)
\\
= (-1)^k   \sum_{e \in \Delta_{k+1}}
\phi_e \Big(\int_{\Omega} u \wedge dz_{e,\emptyset}\Big)
= - \sum_{e \in \Delta_{k+1}}
 \phi_e \Big(\int_{\Omega} du \wedge z_{e,\emptyset}\Big),
\end{multline*}
which implies that $dW^k = W^{k+1}d$.  We will also show below in
Section~\ref{sec:bounds} that $W^k$ can be extended to a bounded
linear operator on $L^2$.

Since we have obtained desired decompositions of the operators $C_n^k
- C_{n-1}^k$ and $C_0^k$, it remains to decompose the functions
$(C_m^k -C_{m-1}^k)u$ for $1 \le m \le n-1$,
cf. \eqref{telescope-sum}.  In fact, the main contribution of this
paper is to show that these operators admit the representation
\begin{equation}\label{Cmk-ident}
  (C_m^k - C_{m-1}^k)u = \sum_{\substack{f \in \Delta_j\\ j = m,m-1}} K_{m,f}^k u,
  \quad 1 \le m \le n-1,
\end{equation}
cf. Proposition~\ref{prop:Cmk-ident} below, where the functions $K_{m,f}^ku$
have support in $\Omega_f$.  Furthermore, each operator $K_{m,f}^k$
commutes with the exterior derivative and preserves piecewise
smoothness and the piecewise polynomial spaces $\P_r\Lambda^k(\T)$ and
$\P_r^-\Lambda^k(\T)$. Our derivation below of the identity
\eqref{Cmk-ident} depends on a careful construction of the family of
operators $\{R_{e,f}^k\}$, or more precisely of the functions
$\{z_{e,f}\}$ defining these operators, cf. \eqref{form-R}.  As a
consequence of the identity \eqref{Cmk-ident}, we obtain the desired
decomposition \eqref{main-decomp}, since the function $u - W^ku$ can
be decomposed into local bubbles of the form
\[
u - W^k u =  \sum_{f \in \Delta} B_{f}^k u,
\]
where each function $B_f^k u$ has support on $\Omega_f$. More precisely,
if we let $K_{n,f}^k = 0$ for each $f \in \Delta_{n-1}$, then 
 \begin{equation}\label{def-B}
B_{f}^k  = K_{m,f}^k + K_{m+1,f}^k, \quad f \in \Delta_m, \, 0 \le m \le n-1,
\end{equation}
while for $f \in \Delta_n$,
 \begin{equation}\label{def-B-n}
B_f^k u = \tr_f (u - C_{n-1}^k u) = \tr_f
\Big(u - (W^k u+ \sum_{\substack{g \in \Delta_m\\0 \le m \le n-1}}B_g^k u)\Big).
\end{equation}
We can summarize the main results we will obtain for the
  construction outlined above in the following theorem.

\begin{thm}\label{thm:main}
For $0 \le k \le n$, there exist operators $W^k :\Lambda^k(\T) \to
\P_1^-\Lambda^k(\T)$ and for each $f \in \Delta_m(\T)$, $0 \le m \le
n$, local operators $B_f^k : \Lambda^k(\T) \to \0\Lambda^k(\T_f)$ such
that the decomposition \eqref{main-decomp} holds.   The
  operators $B_f^k$ can be extended to bounded operators from
  $L^2\Lambda^k(\Omega_f)$ to itself if $0 \le m <  n$, and from
  $L^2\Lambda^k(\Omega_f^E)$ to $L^2\Lambda^k(\Omega_f)$ when $m=n$.
  Furthermore, all these operators commute with the exterior
  derivative, and satisfy the invariant property
  \eqref{space-preserve}.  Finally, the decomposition given by
  \eqref{main-decomp} satisfies the stable decomposition property,
  detailed in Proposition~\ref{prop:L2bound}.
\end{thm}

\section{The case of scalar-valued functions}
\label{sec:scalar}
To motivate the general theory developed later in this paper, we will
discuss the decomposition \eqref{main-decomp} in the case $k=0$. In
fact, to derive the decomposition \eqref{main-decomp}, we only need to
establish the identity \eqref{Cmk-ident} for $1 \le m \le n-1$. A key
ingredient in the derivation of \eqref{Cmk-ident} is to rely on two
slightly different representations of the operator $C_m^0$.  For
$k=0$, the expression \eqref{Cm-rewritten} can be written as
\begin{equation}\label{Cm0-2}
C_m^0 u =\sum_{f \in \Delta_m} 
\sum_{g \in \bar \Delta(f)}(-1)^{|f| -|g|}\Big[L_g^* A_f^0 u 
-  \sum_{i \in I(f \cap g^*)}
 \frac{\lambda_i}{\rho_g}  L_{g}^* A_f^0u \Big],
\end{equation}
where we recall that the operator $R_{e,f}^k u = - A_{\<e,f\>}^k u$
when $(e,f) \in \Delta_{0,m-1}$.  Alternatively, we also have
\begin{equation}\label{Cm0-1}
C_m^0 u =\sum_{f \in \Delta_m} 
\sum_{g \in \bar \Delta(f)}(-1)^{|f| -|g|} \frac{\rho_f}{\rho_g}L_g^* A_f^0 u
:= \sum_{f \in \Delta_m}  C_{m,f}^0 u.
\end{equation}
cf. \eqref{C-primal}.

If we fix $f \in \Delta_m$, then it follows from the definition of the
average operators $A_f^0$ that the function $A_f^0 u$ is defined on
the standard simplex $\S_f^c$, and only depends on $u$ restricted to
$\Omega_f$. In particular, the value of $A_f^0 u$ at the origin is a
weighted average of $u$ over $\Omega_f$. Furthermore, the function
$C_{m,f}^0 u$ will have support on $\Omega_f$. The latter property
follows from a standard cancellation argument. To see this, consider
an index $i \in I(f)$. For each $g \in \Delta(f)$ such that $i \notin
I(g)$, we have that the two terms
\begin{equation*}
    \frac{\rho_f}{\rho_g} L_g^* A_f^0 u, \quad \text{and}
    \quad \frac{\rho_f}{\rho_{\<x_i,g\>}}L_{\<x_i,g\>}^* A_f^0 u
\end{equation*}
cancel, when $\lambda_i(x) =0$. By repeating this argument, we see
that $C_{m,f}^0 u \equiv 0$ for all $x$ such that $\lambda_i(x) = 0$
for all $i \in I(f)$. However, this means that the function $C_{m,f}^0
u$ has support on $\Omega_f$. If $f \in \Delta_0$, i.e., $f = [x_0]$
is a vertex, then
\begin{equation}
\label{Cm0-0}
    (C_{0,f}^0 u)(x) = (A_f^0 u)(\lambda_0(x)) - (1 -\lambda_0(x))(A_f^0 u)(0).
\end{equation}
This operator preserves piecewise smoothness and the piecewise
polynomial spaces, cf. \eqref{A-prop}. On the other hand, if $f \in
\Delta_1$, say $f = [x_0,x_1]$, then
\begin{multline*}
  (C_{1,f}^0u)(x) = (A_fu)(\lambda_0(x), \lambda_1(x))
  \\
  - \frac{\rho_f(x)}{1 -\lambda_1(x)} (A_fu)(0, \lambda_1(x))
  - \frac{\rho_f(x)}{1 -\lambda_0(x)} (A_fu)(\lambda_0(x),0)
  + \rho(x) (A_fu)(0, 0).
\end{multline*}
In contrast to the operator $C_{0,f}^0$, defined by \eqref{Cm0-0}, the
operator $C_{1,f}^0$ will not preserve piecewise smoothness. In fact,
due to the appearance of rational coefficients, the operators
$C_{m,f}^0$ for $1 \le m < n$ will in general not preserve piecewise
smoothness. In the construction performed in \cite{bubble-I}, cf. also
\cite{additive-schwarz}, this problem was overcomed by applying the
operators $C_{1,f}^0$ to the residuals $u - C_0^0 u$. This function is
zero at each vertex, and as a consequence, the operators $C_{1,f}^0(I
- C_0^0)$ will preserve piecewise smoothness.  This led to the
sequential procedure used in \cite{bubble-I}, where we start by
constructing the vertex bubbles and then go on to simplexes of higher
dimension by applying each operator $C_{m,f}^0$ to the proper
residuals.  However, to design a decomposition that works well also
for $k$-forms with $k>0$, we will deviate from this construction, even
in the case of zero-forms.  We will consider the complete effect of
all the operators $C_{m,f}^0$ for $f \in \Delta_m$, and show that the
sum, $C_m^0$, will indeed preserve piecewise smoothness and the
piecewise polynomial spaces.

Motivated by the first term in \eqref{Cm0-2}, we will define the
operators $K_{m,f}^0$ by
\begin{equation}
  \label{def-Kmf0}
K_{m,f}^0u = 
\sum_{g \in \bar \Delta(f)}(-1)^{|f| -|g|} L_g^* A_f^0 u, \quad f \in \Delta_m(\T).
\end{equation}
This operator will have domain of dependence $\Omega_f$, preserve
piecewise smoothness and the piecewise polynomial spaces, and have
support on $\Omega_f$, where the latter property follows by a
cancellation argument as above.
  
From \eqref{Cm0-2}, it follows that
\[
C_m^0 u - \sum_{f \in \Delta_m} K_{m,f}^0 u
= - \sum_{g \in \bar \Delta} \sum_{\substack{f \in \Delta_m\\f \supset g}}(-1)^{|f| -|g|}
\sum_{i \in I(f \cap g^*)} 
\frac{\lambda_i}{\rho_g} L_g^* A_{f}^0 u.
\]
However, for each fixed $g \in \bar \Delta$, we have 
\[
\sum_{\substack{f \in \Delta_m\\f \supset g}}\sum_{i \in I(f \cap g^*)} 
= \sum_{i \in I(g^*)} \sum_{\substack{f \in \Delta_m\\f \supset x_i,g}}
=  \sum_{i \in I(g^*)}\sum_{\substack{f^\prime \in \Delta_{m-1}(x_i^*)\\f^\prime \supset g}}
= \sum_{\substack{f^\prime \in \Delta_{m-1} \\f^\prime \supset g}}\sum_{i \in I((f^\prime)^*)},
\]
where we have introduced $f^\prime = f(\hat x_i)$. As a consequence, we obtain
\[
C_m^0 u - \sum_{f \in \Delta_m} K_{m,f}^0 u
= \sum_{f  \in \Delta_{m-1}}\sum_{g \in \bar \Delta(f)}(-1)^{|f| -|g|} 
\sum_{i \in I(f^*)}\frac{ \lambda_i}{ \rho_g} L_g^* A_{\<x_i,f \>}^0 u.
\]
Furthermore, from \eqref{def-zf} and \eqref{Cm0-1}, we also have
\[
C_{m-1}^0u = \sum_{f \in \Delta_{m-1}} 
\sum_{g \in \bar \Delta(f)}(-1)^{|f| -|g|} 
\frac{\rho_f}{|f^*|\rho_g}\sum_{i \in I(f^*)}L_g^* A_{\<x_i,f\>}^0 u.
\]
By combining these representations of $C_m^0$ and $C_{m-1}^0$, we
obtain the identity
\begin{multline}\label{decomp-pre-0}
(C_m^0 - C_{m-1}^0)u - \sum_{f \in \Delta_m} K_{m,f}^0 u\\
= \sum_{f \in \Delta_{m-1}}\sum_{g \in \bar \Delta(f)}(-1)^{|f| -|g|} \rho_g^{-1}
\sum_{i \in I(f^*)}\Big[\Big(\lambda_i
  - \frac{\rho_f}{|f^*|}\Big) L_g^* A_{\<x_i,f\>}^0 u\Big].
\end{multline}
Therefore, if we define the operators $K_{m,f}^0$ by 
\begin{equation}\label{def-Kmg0}
  K_{m,f}^0 u = \sum_{g \in \bar \Delta(f)} (-1)^{|f| -|g|} \rho_g^{-1} \sum_{i \in I(f^*)}
  \Big[\Big(\lambda_i - \frac{\rho_f}{|f^*|} \Big)  L_g^*A_{\<x_i,f\>}^0 u\Big], 
\end{equation}
for $f \in \Delta_{m-1}$, we obtain that \eqref{decomp-pre-0} simply
reads
\begin{equation}\label{decomp-0}
(C_m^0 - C_{m-1}^0)u = \sum_{\substack{f \in \Delta_j\\ j = m,m-1}} K_{m,f}^0 u,
\end{equation}
which is the desired identity \eqref{Cmk-ident} for $k=0$.

It remains to see that the operators $K_{m,f}^0$, $ f \in
\Delta_{m-1}$, have the desired properties.  Again, by the properties
of the map $A_f^0$, the operators $L_g^* A_{\<x_i,f\>}^0$ and hence
$K_{m.f}^0$ will have domain of dependence $\Omega_f$. In addition, we
need to show that the operators $K_{m,f}^0$, given by
\eqref{def-Kmg0}, preserve piecewise smoothness and piecewise
polynomials, and that the target function has local support.  In fact,
the property that the functions $K_{m,f}^0u$ are supported on the
macroelements $\Omega_f$ follows by a cancellation argument similar to
the one given above. However, formula \eqref{def-Kmg0} contains
rational functions, and therefore, at first glance, it seems unlikely
that the corresponding operators $K_{m,f}^0$ will preserve piecewise
smoothness.  On the other hand, since $\supp K_{m,f}^0 u \subset
\Omega_f$, we can restrict the analysis of the functions $K_{m,f}^0 u
$ to the domain $\Omega_f$, and on this domain we will rely on an
alternative representation of the operators.

To derive the alternative representation of $K_{m,f}^0$, we recall
that when $f \in \Delta_{m-1}$, the manifold $f^*$ is of dimension
$n-m$, and since we assume that $m \le n-1$, the manifold $f^*$ is of
dimension greater or equal to one. For $x_i \in f^*$, we define a
corresponding piecewise linear function, $\beta_{x_i} =
\beta_{x_i}(f)$ by
\[
\beta_{x_i} = \lambda_i - \frac{\rho_f}{|f^*|}. 
\]
The collection $\{\beta_{x_i}(f)\}_{i \in I(f^*)}$ can be seen as an
element in $\C_0(f^*) \otimes \P_1(\T)$.  Furthermore, on the domain
$\Omega_f$, we have that
\[
\partial_0 \beta = \sum_{i \in I(f^*)}\beta_{x_i}
= \Big(\sum_{i \in I(f^*)}
\lambda_i \Big)- \rho_f = 0,
\]
where $\partial_0 = \partial_0(f^*)$. From Lemma~\ref{lem:exact-f*},
we therefore conclude that there is a unique element $\mu = \mu(f) =
\{ \mu_e(f) \}_{e \in \Delta_1(f^*)} \in \C_1(f^*)\otimes \P_1(\T)$,
such that the identities
\begin{equation}\label{mu-rel-0}
(\partial \mu)_{x_i} 
= \lambda_i -  \frac{\rho_f}{|f^*|} = \beta_{x_i},
\quad i \in I(f^*), \quad \text{and } \delta \mu = 0,
\end{equation}
hold on $\Omega_f$, where $\partial = \partial_1(f^*)$ and $\delta =
\delta_1(f^*)$. Note that when $n-m =1$, there is no space $\C_2(f^*)$
(see Figure~\ref{fig:3d-edge-macroelement}) and hence we define
$\delta_1(f^*)$ by \eqref{spec-op}.

We can use the piecewise linear functions, $\{\mu_e(f)\}_{e \in
  \Delta_1(f^*)}$, to obtain an alternative representation of the
operators $K_{m,f}^0 $.  In fact, we have, using \eqref{sum-by-parts}
and \eqref{Rkd-delta-spec}, that
\begin{multline*}
 \sum_{i \in I(f^*)} (\partial \mu )_{x_i} \wedge L_g^*A_{\<x_i,f\>}^0 u
 = - \sum_{i \in I(f^*)} (\partial \mu )_{x_i} \wedge L_g^*R_{x_i,f}^0 u\\
 = - \sum_{e \in \Delta_1(f^*)} \mu_e \wedge L_g^*(\delta R^0u)_{e,f}
 = \sum_{e \in \Delta_1(f^*)} \mu_e \wedge L_g^* R_{e,f}^1 du.
 \end{multline*}
As a consequence, we obtain that the alternative representation,
\begin{equation*}
  K_{m,f}^0 u 
  =  \sum_{g \in \bar \Delta(f)} (-1)^{|f| -|g|} \sum_{e \in \Delta_1(f^*)}
  \Big[\mu_{e}(f) \wedge L_g^* b^{-1}R_{e, f}^1du\Big], 
\end{equation*}
holds on $\Omega_f$.  However, from this alternative representation,
the mapping properties of the operators $K_{m,f}^0$ follow directly
from the corresponding properties of the operators $R_{e,f}^1$, given
in Lemma~\ref{lem:R-pol-prop}.

We close the discussion given above by summarizing the
results we have obtained for the operators $K_{m,f}^0$.
\begin{lem}\label{lem:prop-Kmg0}
Let $1 \le m \le n-1$ and $f \in \Delta_j(\T)$, where $j=m,m-1$.
Assume that the corresponding operators $K_{m,f}^0$ are defined by
\eqref{def-Kmf0} and \eqref{def-Kmg0}.  Then the identity
\eqref{decomp-0} holds, $\supp K_{m,f}^0 u \subset \Omega_f$, and the
operators $K_{m,f}^0$ have the mapping properties
\[
K_{m,f}^0(\Lambda^0(\T_f)) \subset \Lambda^0(\T_f),
\qquad
K_{m,f}^0(\P_r(\T_f) ) \subset \P_r(\T_f).
\]
\end{lem}
Finally, returning to the decomposition for $u \in \Lambda^0$, when
$n=3$, we have seen we can write
\begin{equation*}
      u = W^0 u + \sum_{f \in \Delta_0} B_f^0 u + \sum_{f \in \Delta_1} B_f^0 u
    + \sum_{f \in \Delta_2} B_f^0 u + \sum_{f \in \Delta_3} B_f^0 u,
\end{equation*}
where $W^0u$ is the  piecewise linear function
\begin{equation*}
  W^0u =   \sum_{f \in \Delta_0} \lambda_0(x)  A_f^0 u(0)
    = \sum_{f \in \Delta_0} \lambda_0(x) \int_{\Omega} u \wedge z_f.
\end{equation*}
Here $z_f$ is defined by \eqref{def-zf-n} and \eqref{def-zf}.
Furthermore, from \eqref{def-B} and \eqref{def-B-n},
  \begin{align*}
    B_f^0 u &= K_{0,f}^0 u + K_{1,f}^0u, \qquad f \in \Delta_0,
    \\
    B_f^0 u &= K_{1,f}^0 u + K_{2,f}^0u, \qquad f \in \Delta_1,
    \\
    B_f^0 u &= K_{2,f}^0 u, \qquad f \in \Delta_2,
    \\
    B_f^0 u &= \tr_f (u - [W^0u + \sum_{g \in \Delta_0} B_g^0u
      + \sum_{g \in \Delta_1} B_g^0u + \sum_{g \in \Delta_2} B_g^0u]),
    \qquad f \in \Delta_3.
  \end{align*}
  
\section{The local structure of the mesh}
\label{sec:mesh-functions}
Key tools for decomposing the operators $C_m^k- C_{m-1}^k$ into a sum
of local operators with local target space will be various local
functions derived from the given mesh, $\T$. To describe these
functions, we introduce the space $\P_1^-\Lambda^k(\T,f^*)$ as the
subset of $\P_1^-\Lambda^k(\T)$ corresponding to degrees of freedom on
$f^*$. More precisely,
\[
\P_1^-\Lambda^k(\T,f^*) = \Span_{e \in \Delta_k(f^*)} \phi_e.
\]

\subsection{The functions $\mu_e(f)$}
\label{sec:mu-functions}
Let $m$ be an index, $1 \le m \le n$, and $f \in \Delta_{m-1}$, such
that the associated manifold, $f^*$, will be of dimension $n-m$.  In
the special case when $m=n$, the manifold $f^*$ will be of dimension
zero, and consist of one or two points depending on the location of
$f$ relative to the boundary of $\Omega$.  Since this case is special,
we will first assume that $m \le n-1$ such that the manifold $f^*$ is
a least one dimensional.

In the previous section, we constructed functions $\{\mu_e(f)\}_{e \in
  \Delta_1(f^*)}$ such that the identity \eqref{mu-rel-0} holds, where
we can consider the collection $\{\mu_e(f)\}$ as an element of
$\C_1(f^*) \otimes \P_1^-\Lambda^0(\T,f^*)$.  In fact, we will
construct the collection $\{\mu_e(f) \}_{e \in \Delta_j(f^*)} \in
\C_j(f^*) \otimes \P_1^-\Lambda^{j-1}(\T,f^*)$ for $1 \le j \le n$.
This will be done by an inductive process with respect to $j$.  For $e
\in \Delta_0(f^*)$, we define $\mu_e(f)$ to be the constant
$-1/|f^*|$, such that $\{\mu_e(f) \}_{e \in \Delta_0(f^*)}$ can be
viewed as an element of $\C_0(f^*) \otimes
\P_1^-\Lambda^{-1}(\T,f^*)$, where $\P_1^-\Lambda^{-1}(\T,f^*)$ is
identified as $\R$.  Below we will apply the difference operators,
$\partial(f^*)$ and $\delta(f^*)$, to elements of $\C_j(f^*) \otimes
\P_1^-\Lambda^k(\T,f^*)$. This is done with respect to $\C_j(f^*)$,
while the polynomial space $\P_1^-\Lambda^k(\T,f^*)$ is considered
fixed. For example, the operator $\partial(f^*)$ maps elements of
$\C_j(f^*) \otimes \P_1^-\Lambda^k(\T,f^*)$ to $\C_{j-1}(f^*) \otimes
\P_1^-\Lambda^k(\T,f^*)$. On the other hand, the exterior derivative,
$d$, will map elements of $\C_j(f^*) \otimes
\P_1^-\Lambda^{k}(\T,f^*)$ to $\C_j(f^*) \otimes
\P_1^-\Lambda^{k+1}(\T,f^*)$.

In addition to the collection of functions $\{\mu_e(f)\}$, we will
introduce the associated collection of functions $\{\beta_e\}=
\{\beta_e(f)\}_{e \in \Delta_j(f^*)}$ as elements of $\C_j(f^*)
\otimes \P_1^-\Lambda^j(\T)$ for $0 \le j \le n-m$. The function
$\beta_e(f)$ is defined from the corresponding function $\mu_e(f)$ by
\begin{equation}\label{def-beta}
  \beta_e := \rho_f^{j+1} d\Big( \frac{\mu_e}{\rho_f^j}\Big)
  + (-1)^j \phi_e, \quad e \in \Delta_j(f^*),
\end{equation}
if $1 \le j \le n-m$.  In fact, if we let the exterior derivative
$d_{-1}$ be the inclusion map from constants to
$\P_1\Lambda^{0}(\T,f^*)$, then the definition \eqref{def-beta} also
holds when $j=0$, cf.  \eqref{mu-rel-0}.  We observe that if we
restrict to the domain $\Omega_f$, such that $\rho_f = \sum_{i \in
  I(f^*)} \lambda_i$, then $\beta_e(f)$ also admits the representation
\begin{equation}\label{def-beta-alt}
  \beta_e = \sum_{i \in I(f^*)} (\lambda_i d - j d\lambda_i \wedge)\mu_e
  +(-1)^j \phi_e, \quad e \in \Delta_j(f^*).
\end{equation}
It then follows that
\begin{equation*}
  \mu_e \in \P_1^-\Lambda^{j-1}(\T, f^*) \,
  \Longrightarrow \, \tr_{\Omega_f} \beta_e
  \in \P_1^-\Lambda^{j}(\T_{f}, f^*),
\end{equation*}
where $\P_1^-\Lambda^{j}(\T_{f}, f^*) = \tr_{\Omega_f}
\P_1^-\Lambda^{j}(\T, f^*)$. Furthermore, the relation
\eqref{mu-rel-0} can be rephrased as $\beta_e = (\partial(f^*) \mu)_e$
on $\Omega_f$ for $e \in \Delta_0(f^*)$.

\begin{lem}
\label{lem:mu-1}
Let $f \in \Delta_{m-1}(\T)$, where $1 \le m \le n-1$, and $j$ an
index such that $1 \le j \le n-m$. Assume that $\{\mu_e\} =
\{\mu_e(f)\} \in \C_s(f^*) \otimes \P_1^-\Lambda^{s-1}(\T, f^*)$ has
been defined for $s= j-1, j$ such that the identity
\begin{equation}\label{mu-rel}
\beta_e(f) = (\partial(f^*) \mu)_e,
\end{equation}
holds on $\Omega_f$ for $e \in \Delta_{j-1}(f^*)$.  Then
$(\partial(f^*)\beta)_e = 0$ for all $e \in \Delta_{j-1}(f^*)$, and if
$j <n-m$, there exist $\{\mu_e\} = \{\mu_e(f) \} \in \C_{j+1}(f^*)
\otimes \P_1^-\Lambda^{j}(\T,f^*)$ such that \eqref{mu-rel} holds on
$\Omega_f$ for all $e \in \Delta_j(f^*)$.  Furthermore, $\{\mu_e\} \in
C_{j+1}(f^*) \otimes \P_1^-\Lambda^{j}(\T,f^*)$ is uniquely determined
by the condition $\delta(f^*) \mu =0$.
\end{lem}

\begin{proof} Throughout the proof,
  all the identities should be considered to hold on the domain
  $\Omega_f$, and the operators $\partial$ and $\delta$ are defined
  with respect to $f^*$. By assumption, we have
\[
\beta_e = \rho^j d\Big(\frac{\mu_e}{\rho^{j-1}}\Big) + (-1)^{j-1} \phi_e
= (\partial_j \mu)_e, \quad e \in \Delta_{j-1}(f^*),
\]
where $\rho= \rho_f$.
Since $d^2 = 0$ and $d$ commutes with $\partial_j$, this gives
\[
(\partial_j d\Big(\frac{\mu}{\rho^j}\Big))_e
=  (-1)^{j-1} d \Big(\frac{\phi_e}{\rho^j}\Big),
\quad e \in \Delta_{j-1}(f^*).
\]
Hence, it follows from \eqref{def-beta} that for $e \in \Delta_{j-1}(f^*)$
\[
(\partial_j \beta)_e = \rho^{j+1}(\partial_j d\Big(\frac{\mu}{\rho^j}\Big))_e
  + (-1)^j (\partial_j\phi)_e 
    = (-1)^{j-1}[\rho^{j+1}d \Big(\frac{\phi_e}{\rho^j}\Big)
      -  (\partial_j\phi)_e].
\]
However, a direct computation, using \eqref{phi-xf} and $\rho_f =
\sum_{i \in I(f^*)}\lambda_i$ on $\Omega_f$, shows
\[ \rho^{(j+1)}d \Big(\frac{\phi_e}{\rho^j}\Big) 
= \sum_{i \in I(f^* \cap e^*)} (\lambda_i d - j d\lambda_i \wedge )\phi_e
=    \sum_{i \in I(f^*\cap e^*)} \phi_{[x_i,e]} =  (\partial_j \phi)_e.
\]
As a consequence, $(\partial_j \beta)_e = 0$ for $e \in
\Delta_{j-1}(f^*)$.  If $j <n-m$, it follows from
Lemma~\ref{lem:exact-f*} that there exist a uniquely determined
$\{\mu_e\} \in \C_{j+1}(f^*) \otimes \P_1^-\Lambda^j(\T_f,f^*)$ such
that $\beta_e = (\partial_{j+1}\mu)_e$ on $\Omega_f$, for all $e \in
\Delta_j(f^*)$ and $\delta_{j+1} \mu = 0$.  Furthermore, each function
$\mu_e \in \P_1^-\Lambda^j(\T_f,f^*)$ can be uniquely extended to a
function in $\P_1^-\Lambda^j(\T,f^*)$.
\end{proof}

It remains to discuss the case $m=n$. In this case, $f \in
\Delta_{n-1}$, so $f^*$ will only consist of one or two points, and
there is no element in $\Delta_1(f^*)$.  In fact, we will have $|f^*|
=1$ if $f$ is a boundary simplex, and otherwise $|f^*| =2$.  On the
other hand, by adopting the interpretation above of $d_{-1}$ as the
inclusion operator and $\mu_e(f) = -1/|f^*|$ for $e \in
\Delta_0(f^*)$, we obtain
\begin{equation*}
\beta_e(f):= - \frac{\rho_f}{|f^*|} + \phi_e, \quad e \in \Delta_0(f^*),
\end{equation*}
and $\partial_0(f^*) \beta = 0$.

From Lemma~\ref{lem:mu-1} and an induction argument with respect to
$j$, we obtain the following result.
\begin{cor}\label{mu-sum-up}
  Let $f \in \Delta_{m-1}(\T)$, where $1 \le m \le n$. For all $e \in
  \Delta_j(f^*)$, $0 \le j \le n-m$, there exists functions $\mu_e(f)
  \in \P_1^-\Lambda^{j-1}(\T,f^*)$ and $\beta_e(f) \in
  \P_1^-\Lambda^{j}(\T)$, uniquely defined by the
  inductive procedure above, satisfying $\partial_j(f^*) \beta =0$ on
  $\Omega_f$. Furthermore, if $j < n-m$, then the identity
  \eqref{mu-rel} holds on $\Omega_f$ for all $e \in \Delta_j(f^*)$.
\end{cor}

\begin{remark}
  An alternative view of the construction of
    the functions $\{ \mu_e(f)\}$ given above can be given by expanding
    the function $\mu_e \in \P_1^-\Lambda^{j-1}(\T,f^*)$ in the form
\begin{equation}\label{mu-expand}
\mu_e = \sum_{e^\prime \in \Delta_{j-1}(f^*)} a_{e,e^\prime} \phi_{e^\prime},
\end{equation}
where the real coefficients $\{a_{e,e^\prime} \}$ can be identified
with an element $a_e \in \C_{j-1}$.  It follows from \eqref{d-phi},
\eqref{phi-xf},  and \eqref{def-beta-alt},  that if we
restrict to $\Omega_f$, then the function $\beta_e = \beta_e (f)$
admits the representation
\begin{equation}\label{beta-expand}
  \beta_e - (-1)^j  \phi_e
  =  \sum_{\substack{e^\prime \in \Delta_j(f^*)\\ i \in I(e^\prime)}}
  (-1)^{\sigma_{e^\prime}(x_i)} a_{e, e^\prime(\hat x_i)}\phi_{e^\prime}
=  \sum_{e^\prime \in \Delta_j(f^*)} (\delta a_{e, \cdot})_{e^\prime} \phi_{e^\prime}, 
\end{equation}
for $e \in \Delta_j(f^*)$. As a consequence, the equation
\eqref{mu-rel} for $e \in \Delta_j(f^*)$ can be represented by the
algebraic system
\begin{equation}\label{mu-rel-algebraic}
  (\partial_{j+1} a_{\cdot, e^\prime})_e =  (\delta_{j-1} a_{e,\cdot})_{e^\prime}
  + (-1)^j  1_{e,e^\prime}, \quad e^\prime \in \Delta_j(f^*),
\end{equation}
where $\partial = \partial(f^*)$, and $1_{e,e^\prime} =1$ if $e^\prime =e$ and
equals zero otherwise. Furthermore, the condition 
$\delta_{j+1}(f^*) \mu= 0$ is equivalent to
\begin{equation}
  \label{mu-extra}
  \delta_{j+1} a_{\cdot, e^\prime} = 0, \quad 
  e^\prime \in  \Delta_{j}(f^*).
  \end{equation}
\end{remark}

If $(e,f) \in \Delta_{j,m-1}$ and $g \in \bar \Delta(f)$, then $\rho_g
-\rho_f= \sum_{i \in I(f \cap g^*)} \lambda_i$, which leads to
\[
\rho_g^{j+1}d\Big(\frac{\mu_e(f)}{\rho_g^j}\Big)
= \rho_f^{j+1}d\Big(\frac{\mu_e(f)}{\rho_f^j}\Big) +\sum_{i \in I(f \cap g^*)}
\big(\lambda_i d - j d\lambda_i \wedge \big)\mu_e(f).
\]
 In the analysis below, the functions $\psi_{e,g}(f) \in
\P_1^-\Lambda^j(\T,g^*)$, defined by
\begin{equation}
  \label{def-psi}
\psi_{e,g}(f) = (-1)^{j-1}\sum_{i \in I(f \cap g^*)}
\big(\lambda_i d - j d\lambda_i \wedge \big)\mu_e(f),
\end{equation}
will be useful. We observe that $\psi_{e,f}(f) = 0$, and it follows from 
\eqref{def-beta}, that 
\begin{equation}\label{psi-id}
  \rho_g^{j+1}d\Big(\frac{\mu_e(f)}{\rho_g^j}\Big)
  + (-1)^j  [\phi_e + \psi_{e,g}(f)] = \beta_e(f).
\end{equation} 
Also observe that in the special case when $e \in \Delta_0(f^*)$, then 
\begin{equation}\label{psi-spec}
  \psi_{e,g}(f) = \frac{1}{|f^*|} \sum_{i \in I(f \cap g^*)} \lambda_i
  = \frac{1}{|f^*|} (\rho_g- \rho_f).
\end{equation}

\subsection{Construction of the weight functions $z_{e,f}$}
\label{sec:weight-func}
We recall from Section \ref{sec:average} that the weight functions
$z_{e,f}$ are an essential ingredient for the construction of the
order reduction operators $R_{e,f}^k$.  This section is devoted to an
inductive process for constructing the functions $z_{e,f}$.  In fact,
as a preliminary step, we will first construct a family of local
functions $w_{e,f}$, and then the functions $z_{e,f}$ will be
constructed as
\begin{equation}\label{w-to-z}
z_{e,f} = (\delta^+ w)_{e,f}.
\end {equation}
Here the operator $\delta^+$ is a variant of the coboundary operator
defined for pairs of simplices, given by
\[
(\delta^+ w)_{e,f} = \sum_{i \in I(e)} (-1)^{\sigma_e(x_i)} w_{e(\hat x_i),\<x_i, f\>},
\quad (e,f) \in \Delta_{j,m}. 
\]
This operator will satisfy the complex property, $(\delta^+)^2 = 0$,
and from \cite[Lemma 5.2]{bubble-II}, we recall that $\delta \circ
\delta^+ = - \delta^+ \circ \delta$.  For the construction below, we
will utilize exactness of the complex of trimmed linear forms with
support on $\Omega_f$. Recall that in Section~\ref{sec:intro}, we
introduced the space $\0\P_1^-\Lambda^k(\T_f)$ as the subspace of
$\P_1^-\Lambda^k(\T)$ consisting of functions which vanish on $\Omega
\setminus \Omega_f$. However, if $f$ is a boundary simplex, then
functions in this space will in general not have vanishing trace on
the boundary of $\Omega_f$.  Therefore, we introduce the notation
$\dzero \P_1^-\Lambda^k(\T_f)$ to denote the subspace of
$\0\P_1^-\Lambda^k(\T_f)$ with vanishing trace on $\partial
\Omega_f$. The two spaces are equal if $f$ is an interior simplex, but
in general $\dzero \P_1^-\Lambda^k(\T_f) \subset
\0\P_1^-\Lambda^k(\T_f)$.  We also recall from \eqref{span-d} that if
$w = \sum_{g \in \Delta_j} c_g \phi_g \in \P_1^-\Lambda^j(\T)$, then
$dw = \sum_{g \in \Delta_{j+1}} (\delta_jc)_g\phi_g$.  In particular,
if $w$ has support on $\Omega_f$ for $f \in \Delta$, then the sum can
be restricted to all simplices $g$ such that $g\supset f$.  Motivated
by this, we define $d_{f}^* : \dzero\P_1^-\Lambda^j(\T_f) \to
\dzero\P_1^-\Lambda^{j-1}(\T_f)$, by
\begin{equation}\label{df*}
d_f^*w = \sum_{\substack{g \in \Delta_{j-1}\\g \supset f}} (\partial_jc)_g \phi_g. 
\end{equation}
Below we will utilize exactness of the complex $(\dzero\P_1^-\Lambda(\T_f),d)$
to conclude that a function 
$w \in \dzero\P_1^-\Lambda^j(\T_f)$
is uniquely determined by $dw$ and $d_f^*w$.

The functions $w_{e,f}$ will be defined inductively with respect to
$m$ for all pairs $(e,f) \in \Delta_{j,m}$, for $0 \le m \le n$ and
$-1 \le j <n-m$ as functions in $\dzero
\P_1^-\Lambda^{n-j-1}(\T_f)$. We start the induction process with
$m=n$, and hence only $j=-1$ is allowed.  The set $\Delta_{-1,n}$
consists of pairs of the form $(\emptyset, f)$, where $f \in
\Delta_{n}$.  In this case, we define $w_{\emptyset ,f} = -z_f =
-(\kappa_f/|\Omega_f|)\vol$.  For the general case, when $0 \le m <
n$, we will use a variational approach utilizing tensor product spaces
of the form $\P_1^-\Lambda^j(\T,f^*) \otimes
\P_1^-\Lambda^{n-j-1}(\T)$, i.e., we consider products of elements in
the two functions spaces and with independent spatial variables.
We assume, as an induction hypothesis, that $w_{e,f} \in
  \dzero\P_1^-\Lambda^{n-j-1}(\T_f)$ for all $(e,f) \in \Delta_{j,m}$
  and $-1 \le j <n -m$, have already been constructed.  Furthermore,
  we assume that these functions satisfy
\begin{equation}
  \label{ind-hyp}
(\delta^+ dw)_{e,f} \in \text{range}(\delta), \quad (e,f)\in
  \Delta_{n-m,m-1}(f^*).
\end{equation}
In the case $m=n$, the functions of the form
$w_{\emptyset,\<x_i,f\>}$, involved in \eqref{ind-hyp}, have support
in $\Omega_{\<x_i,f\>} \subset \Omega_f$. Since we will utilize
exactness of the complex consisting of trimmed differential forms with
boundary conditions, the exterior derivative in \eqref{ind-hyp}, in
the case $m=n$, should be interpreted as the integral, and for
$(x_i,f) \in \Delta_{0, n-1}$, we have
\begin{equation*}
        (\delta^+ dw)_{x_i,f} = d (\delta^+w)_{x_i,f}
        = \int_{\Omega} w_{\emptyset,\<x_i,f\>} = -1
        = \int_{\Omega} w_{\emptyset,f} = (\delta dw)_{x_i,f}.
\end{equation*}
Therefore, property \eqref{ind-hyp} holds initially.

For a fixed $f \in \Delta_{m-1}$ and $-1 \le j < n-m$, $w_{e,f} \in
\P_1^-\Lambda^{n-j-1}(\T)$ is defined by
\begin{equation}
  \label{F-def}
\sum_{e \in \Delta_{j}(f^*)} \phi_{e} \otimes w_{e,f}
= (-1)^j \sum_{e \in \Delta_{j+1}(f^*)} \mu_{e}(f) \otimes
(\delta^+ w)_{e,f},
\end{equation}
where we observe that all functions on the right hand side are already
constructed.  Since $\mu_e(f) \in \P_1^-\Lambda^j(\T,f^*)$ for $e \in
\Delta_{j+1}(f^*)$, we can view the right hand side of \eqref{F-def}
as an element of $\P_1^-\Lambda^j(\T,f^*) \otimes
\P_1^-\Lambda^{n-j-1}(\T)$, i.e, it is a trimmed linear $j$-form with
values in $\P_1^-\Lambda^{n-j-1}(\T)$. Therefore, the coefficients
$w_{e,f}$ of the left hand side are uniquely determined as functions
in $\P_1^-\Lambda^{n-j-1}(\T)$. Furthermore, since all the domains of
the form $\{\Omega_{x_i,f}\}$, $i \in I(f^*)$ are contained in
$\Omega_f$, we can conclude from the induction hypothesis that
$(\delta^+ w)_{e,f} \in \dzero\P_1^-\Lambda^{n-j-1}(\T_f)$ and hence
the function $w_{e,f} \in \dzero\P_1^-\Lambda^{n-j-1}(\T_f)$ for
$(e,f) \in \Delta_{j,m-1}$.  In particular, since $\mu_e(f) =
-1/|f^*|$ for $e \in \Delta_0(f^*)$, we obtain that the function
$w_{\emptyset,f}$, $f \in \Delta_{m-1}$ satisfies the recurrence
relation
\begin{equation*}
w_{\emptyset ,f} = \frac{1}{|f^*|}
\sum_{i \in I(f^*)} w_{\emptyset ,\<x_i,f\>}.
\end{equation*}
Let $0 \le j \le n-m$. For the discussion below, it will be useful to
observe that the function just defined satisfies
 \begin{align*}
    \sum_{e \in \Delta_j(f^*)} \phi_e \otimes (\delta w)_{e,f}
   &=  \sum_{e \in \Delta_{j-1}(f^*)} (\partial \phi)_e \otimes w_{e,f}
   \\
   &= \sum_{i \in I(f^*)} (\lambda_i d -j d \lambda_i \wedge)
   \sum_{e \in \Delta_{j-1}(f^*)} \phi_e
   \otimes w_{e,f}
   \\
   &=  (-1)^{j-1} \sum_{i \in I(f^*)} (\lambda_i d -j d \lambda_i \wedge)
     \sum_{e \in \Delta_j(f^*)} \mu_e \otimes (\delta^+w)_{e,f},
   \end{align*}
where we have used \eqref{sum-by-parts} and \eqref{phi-xf}, in
addition to \eqref{F-def}. Alternatively, from \eqref{def-beta-alt} we
have
\begin{equation}\label{w-property}
    \sum_{e \in \Delta_j(f^*)}\Big[ \phi_e \otimes (\delta w)_{e,f}
      - \Big(\phi_e + (-1)^{j-1} \beta_e \Big)
   \otimes    (\delta^+w)_{e,f}\Big] = 0,
\end{equation}
where $f \in \Delta_{m-1}$ and $0 \le j \le n-m$, and where the
spatial variable for $\phi_e$ and $\beta_e$ is restricted to
$\Omega_f$.

The set up above defines the functions $w_{e,f}$ for all $(e,f) \in
\Delta_{j,m-1}$, where $-1 \le j <n-m$, from the corresponding
functions defined for elements in $\Delta_{j,m}$.  However, we will
also need the functions $w_{e,f}$ for $(e,f) \in \Delta_{n-m,m-1}$. In
this case, the right hand side of \eqref{F-def} is not well defined
since there are no elements in $\Delta_{n-m,m}$.  Instead, for $f \in
\Delta_{m-1}(\T)$ and $j=n-m$, we will require $w_{e,f}$ to satisfy
\begin{equation}\label{F-def-special}
 \sum_{e \in \Delta_{j}(f^*)} \phi_e \otimes dw_{e,f} =
   \sum_{e \in \Delta_{j}(f^*)} \beta_e(f)
  \otimes (\delta^+w)_{e,f}, 
\end{equation}
and $d_f^* w_{e,f} = 0$.  In \eqref{F-def-special}, the spatial
variables are restricted to $\Omega_f \times \Omega$, i.e., $\phi_e$
and $\beta_e(f)$ mean the restrictions of these quantities to
$\Omega_f$.  It follows from the induction hypothesis and the fact
that $\tr_{\Omega_f}\beta_e(f) \in \P_1^-\Lambda^{n-m}(\T_f,f^*)$,
that the right hand side of \eqref{F-def-special} can be viewed as an
element of $\P_1^-\Lambda^{n-m}(\T_f,f^*) \otimes
\P_1^-\Lambda^{m}(\T)$.  Furthermore, from the hypothesis
\eqref{ind-hyp}, combined with \eqref{sum-by-parts}, we obtain that
\[
 \sum_{e \in \Delta_{j}(f^*)} \beta_e(f)
  \otimes (\delta^+dw)_{e,f}= 0,
\]
since $\partial \beta(f) = 0$.  As a consequence, from the exactness
of the complex $(\dzero\P_1^-\Lambda(\T_f),d)$, we obtain that there
exist elements $w_{e,f} \in \dzero\P_1^-\Lambda^{m-1}(\T_f)$ such that
\eqref{F-def-special} holds.  Finally, since we require $d_f^*w_{e,f}
= 0$, the functions $w_{e,f}$ are uniquely determined.

To complete the induction argument, we need to verify that the
assumption \eqref{ind-hyp} is preserved by the induction step, i.e.,
that the identity \eqref{ind-hyp} holds with $m$ replaced by
$m-1$. However, by combining \eqref{w-property} and
\eqref{F-def-special}, we obtain that the identity
\begin{equation}\label{d-delta-rel}
dw_{e,f} = (-1)^{j+1}[(\delta w)_{e,f} - (\delta^+ w)_{e,f}]
\end{equation}
holds for $(e,f) \in \Delta_{j,m-1}$, $j= n-m$. As a further
consequence of the complex property of $\delta^+$, and the identity
$\delta^+ \circ \delta = -\delta \circ \delta^+$,
we then obtain
\[
(\delta^+ dw )_{e, f} = (-1)^{n-m+1}(\delta \circ \delta^+ w)_{e,f},
\]
for $(e,f) \in \Delta_{n-m+1,m-2}$. Hence, property \eqref{ind-hyp} at
level $m-1$ is verified.

We summarize the properties of the construction above.
\begin{lem}\label{lem:w-def}
The inductive procedure above uniquely specifies the functions
$w_{e,f} \in \dzero\P_1^-\Lambda^{n-j-1}(\T_{f})$ for all $(e,f) \in
\Delta_{j,m}$, where $0 \le m \le n$ and $-1 \le j < n-m$.
\end{lem}
Next we will establish that the functions $w_{e,f}$, introduced above,
satisfy the identity \eqref{d-delta-rel} more generally, i.e., not
only for $(e,f) \in \Delta_{n-m,m-1}$.

\begin{lem}\label{lem:d-delta}
The functions $w_{e,f}$ satisfy the identity \eqref{d-delta-rel} for
all $(e,f) \in \Delta_{j,m}$, where $0 \le m \le n-1$ and $0 \le j <
n-m$.
\end{lem}
\begin{proof}
The proof will be done by induction with respect to $m$.  For $m=n-1$,
the only possible value of $j$ is $j= 0$.  In this case, the desired
identity has already been verified above.  Next, we assume that the
identity holds for all $(e,f) \in \Delta_{j,m}$, where $0 \le j <n-m$.
Note that the case $(e,f) \in \Delta_{n-m,m-1}$ is already established
above. Therefore, we can assume that $j < n-m$, and from the support
property of the functions $w_{e,f}$, it is enough to show this identity
on $\Omega_f$.  It follows from \eqref{w-property} that the desired
identity will follow if we can show that
\begin{equation}\label{beta-sum}
   \sum_{e \in \Delta_j(f^*)} \beta_e \otimes (\delta^+w)_{e,f}
   =  \sum_{e \in \Delta_j(f^*)} \phi_e \otimes d w_{e,f},
   \quad f \in \Delta_{m-1},
\end{equation}
when $\phi_e$ and $\beta_e$ are restricted to $\Omega_f$.  If $0 \le j
<m-n$, then from Corollary~\ref{mu-sum-up} we have the identity
$\beta_e = (\partial \mu)_e$ which gives
\begin{multline*}
   \sum_{e \in \Delta_j(f^*)} \beta_e \otimes (\delta^+w)_{e,f}
   =  \sum_{e \in \Delta_j(f^*)} (\sig \mu)_e \otimes (\delta^+w)_{e,f}
   \\
   = -  \sum_{e \in \Delta_{j+1}(f^*)} \mu_e \otimes (\delta^+\circ \delta w)_{e,f}
   = (-1)^j  \sum_{e \in \Delta_{j+1}(f^*)} \mu_e \otimes (\delta^+ dw)_{e,f}
   \\
   =    \sum_{e \in \Delta_j(f^*)} \phi_e \otimes dw_{e,f},
\end{multline*}
where we have used the fact that $\delta \circ \delta^+= - \delta^+
\circ \delta$, \eqref{F-def}, and the induction hypothesis
\eqref{d-delta-rel} with $f^\prime = \{x_i,f\} \in \Delta_m$ and
$e^\prime = e(\hat x_i) \in \Delta_j((f^\prime)^*)$. This establishes
the identity \eqref{beta-sum}, and therefore the proof is completed.
\end{proof}

The desired weight functions $z_{e,f}$ are defined from the
corresponding $w$ functions by the relation \eqref{w-to-z}.  More
precisely, the functions $z_{e,f}$ are defined by \eqref{w-to-z} for
$(e,f) \in \Delta_{j,m}$ for $-1 \le m <n$ and $0 \le j < n-m$.  In
particular, for $(e,f) \in \Delta_{0,n-1}$, we have $z_{e,f} =
-z_{\<e,f\>}$.  Recall also that a key property of these functions is
that they satisfy the identity \eqref{z-prop}, i.e., $d z =
(-1)^{j+1}\delta z$ for $(e,f) \in \Delta_{j,m}$.
   
\begin{lem}\label{lem:z-prop-1}
The functions $z_{e,f}$, where $(e,f) \in \Delta_{j,m}$ for $-1 \le m
\le n-1$ and $0 \le j < n-m$ belong to $\P_1^-\Lambda^{n-j}(\T)$, and
with support in $\Omega_{f} \cap \Omega_e^E$. In addition, $z_{e,f}$
has vanishing trace on the boundary of $\Omega$.  Furthermore, for $j
>0$, the identities \eqref{z-prop} and $(\delta^+ z)_{e,f} = 0$ hold.
\end{lem}

\begin{proof} 
That the functions $z_{e,f}$ belong to $\P_1^-\Lambda^{n-j}(\T)$
follows from the fact that $w_{e,f}$ are elements of
$\P_1^-\Lambda^{n-j-1}(\T)$. Furthermore, since $\delta^+$ satisfies the
complex property, we obtain that $\delta^+ z = 0$, and from
\eqref{d-delta-rel} we have
\[
  (d z)_{e,f} =  (\delta^+ d w)_{e,f}
= (-1)^j (\delta^+ \circ \delta w)_{e,f}
\\  =  (-1)^{j+1}(\delta \circ \delta^+ w)_{e,f}
= (-1)^{j+1} (\delta z)_{e, f},
\]
where, as above, we have used the fact that
$\delta \circ \delta^+ = - \delta^+ \circ \delta$.

From the support property of the $w$ functions given in
Lemma~\ref{lem:w-def}, we have for $(e,f) \in \Delta_{j,m}$ and $i \in
I(e)$ that
\[
\supp w_{e(\hat x_i),\<x_i,f\>} \subset \Omega_{\<x_i,f\>}
= \Omega_f \cap \Omega_{x_i},
\]
which implies that 
\[
\supp z_{e,f} \subset 
\bigcup_{i \in I(e)} \big(\Omega_f \cap \Omega_{x_i}\big)= \Omega_f \cap \Omega_e^E.
\]
Finally, since all functions $\{w_{e,f}\}$ have vanishing trace on the
boundary of $\Omega$, this property will also hold for all the
functions $\{z_{e,f}\}$.
\end{proof}
  
Recall that the functions $\psi_{e,g}(f)$, defined by \eqref{def-psi},
where $(e,f) \in \Delta_{j,m-1}$ and $g \in \bar \Delta(f)$, satisfy
the relation \eqref{psi-id}. The following relation between the
functions $\psi_{e,g}(f)$ and the weight functions $z_{e,f}$ will be
crucial in our analysis below.

\begin{lem}
   \label{psi-prop}
   Let $0 \le m \le n-1$ and $0 \le j < n-m$. For any $g \in
   \Delta_s(\T)$, where $-1 \le s \le m-1$, the identity
\begin{equation}\label{residual-idg}
\sum_{\substack{(e,f) \in \Delta_{j,m}\\ f \supset g}}
\psi_{e,g}(f) \otimes z_{e,f}
  = 
  \sum_{\substack{(e,f) \in \Delta_{j,m-1}\\ f \supset g}}
   \phi_e \otimes z_{e,f},
\end{equation}
holds.
\end{lem}
\begin{proof}
Since $z = \delta^+w$, we can reformulate the right hand side of the
identity as
\[
\sum_{\substack{(e,f) \in \Delta_{j,m-1}\\ f \supset g}}
\sum_{i \in I(e)}
(-1)^{\sigma_{e}(x_i)} \phi_e \otimes w_{e(\hat x_i), \<x_i,f\>}
= \sum_{\substack{(e,f) \in \Delta_{j-1,m}\\ f \supset g}} \sum_{i \in I(f \cap g^*)}
 \phi_{[x_i,e]} \otimes w_{e,f}.
\]
On the other hand, using the definition of $\psi_{e,g}(f)$,
\eqref{phi-xf}, and \eqref{F-def}, for $e \in \Delta_{j-1}$, the left
hand side can be rewritten as
\begin{multline*}
(-1)^{j-1}\sum_{\substack{f \in \Delta_{m}\\ f \supset g}}
  \sum_{i \in I(f \cap g^*)}
  (\lambda_i d - j d\lambda_i \wedge)
  \sum_{e \in \Delta_j(f^*)} \mu_e \otimes (\delta^+w)_{e,f}
  \\
 =  \sum_{\substack{f\in \Delta_{m}\\ f \supset g}}
\sum_{i \in I(f\cap g^*)}(\lambda_i d - j d\lambda_i \wedge)
\sum_{e \in \Delta_{j-1}(f^*)} \phi_e\otimes w_{e,f}\\
= \sum_{\substack{(e,f) \in \Delta_{j-1,m}\\ f \supset g}} \sum_{i \in I(f \cap g^*)}
 \phi_{[x_i,e]} \otimes w_{e,f},
\end{multline*}
and hence the desired identity is verified.
\end{proof}

\section{The order reduction operators}
\label{sec:order-reduct}
We recall from Section~\ref{sec:average} above that the order
reduction operators, $R_{e,f}^k$, are defined for $(e,f) \in
\Delta_{j,m}$ from corresponding functions $z_{e,f}$ by
\[
(R_{e,f}^k u)_{\lambda}
=  \int_{\Omega} (\Pi_j G_f^*u)_{\lambda} \wedge z_{e,f}, \quad 
\lambda \in \S_f^c,
\]
cf. \eqref{form-R}. The functions $z_{e,f}$ will be taken to be the
weight functions constructed in the previous section.  In particular,
$z_{e,f} \in \P_1^-\Lambda^{n-j}(\T_f)$ for $(e,f) \in \Delta_{j,m}$,
and with support in $\Omega_f \cap \Omega_e^E$, while the function
$G_f$, defined above, maps the product $\Omega_f \times \S_f^c$ to
$\Omega_f$.  Since $G_f$ is defined on a product space, the target
space for the pullback, $G_f^*$, can be represented as the sum of
tensor products, and $\Pi_j$ is the canonical map of
$\Lambda^k(\Omega_f \times \S_f^c)$ to $\Lambda^j(\Omega_f) \otimes
\Lambda^{k-j}(\S_f^c)$.  By construction, the operators $R_{e,f}^k$
map $k$-forms to $(k-j)$-forms for $(e,f) \in \Delta_{j,m}$, $0 \le j
\le k$. Furthermore, $R_{e,f}^k \equiv 0$ for $j > k$.

The commuting property $dG_f^*u = G_f^* du$, where $u$ is $k$-form on
$\Omega_f$, can in the present setting be expressed by
\begin{equation}
\label{dG*}
d_{\Omega} \Pi_{j-1} G_f^*u  + (-1)^{j} d_S \Pi_j G_f^*u = \Pi_j d G_f^*u = 
\Pi_j G_f^* du, \quad j=1, \ldots, k,
\end{equation} 
where $d_{\Omega}$ and $d_{S}$ denote the exterior derivative with
respect to the spaces $\Omega$ and $\S_f^c$, respectively.  By
combining this with property \eqref{z-prop},
cf. Lemma~\ref{lem:z-prop-1}, we derive, exactly as in the proof of
\cite[Proposition 5.4]{bubble-II}, that the operators $R_{e,f}^k$
satisfy the fundamental relation \eqref{Rkd-delta}.  Furthermore, from
the fact that $\delta^+ z = 0$, cf. Lemma~\ref{lem:z-prop-1}, we
obtain that $(\delta^+ R^k)_{e,f} = 0$.

Next we will use Lemma~\ref{psi-prop} to obtain the following property
of the operators $R_{e,f}^k$.
\begin{lem} \label{lem:R-psi-prop}
 Let $0 \le m \le n-1$ and $0 \le j <
n-m$. For any  $g \in \Delta_s(\T)$, where $-1 \le s \le m-1$, the identity
\begin{equation}\label{residual-idRg}
\sum_{\substack{(e,f) \in \Delta_{j,m}\\ f \supset g}}
  \psi_{e,g}(f)   \wedge L_g^* R_{e,f}^k u
 = \sum_{\substack{(e,f) \in \Delta_{j,m-1}\\ f \supset g}}
 \phi_e   \wedge L_g^* R_{e,f}^k u,
\end{equation}
holds.
\end{lem}
\begin{proof} We consider the simplex $g$ to be fixed.
Since $\psi_{e,g}(f) \in \P_1^-\Lambda^j(\T,g^*)$, it follows that
there exist constants $\{a_{e,e^\prime}(f)\}$ such that $\psi_{e,g}(f)
= \sum_{e^\prime \in \Delta_j(g^*)} a_{e,e^\prime}(f)
\phi_{e^\prime}$. As a consequence, the identity \eqref{residual-idg}
can be expressed as
\[
\sum_{\substack{(e,f) \in \Delta_{j,m-1}\\ f \supset g}}\phi_e \otimes z_{e,f}
    = \sum_{e \in \Delta_j(g^*)} \phi_e \otimes
   \sum_{\substack{(e^\prime,f) \in \Delta_{j,m}\\ f \supset g}}
 a_{e^\prime ,e}(f) z_{e^\prime,f},
\]
which implies that 
\[
\sum_{\substack{f \in \Delta_{m-1}(e^*)\\ f \supset g}} z_{e,f}
= \sum_{\substack{(e^\prime,f) \in \Delta_{j, m}\\ f \supset g}}
 a_{e^\prime ,e}(f) z_{e^\prime,f}, \quad e \in \Delta_j(g^*).
 \]
From the definition of the order reduction operators $R_{e,f}^k$, we
then obtain the corresponding identity
\[
\sum_{\substack{f \in \Delta_{m-1}(e^*)\\ f \supset g}} R_{e,f}^k
=  \sum_{\substack{(e^\prime,f) \in \Delta_{j, m}\\ f \supset g}}
 a_{e^\prime ,e}(f) R_{e^\prime,f}^k, \quad e \in \Delta_j(g^*).
 \]
By applying $L_g^*$ and then reversing the steps above, we obtain the
 desired identity.
\end{proof}

In addition to the operators $R_{e,f}^k$, we will also use the
operators $Q_{e,f}^k$, defined from the functions $w_{e,f}$. More
precisely,
 \[
(Q_{e,f}^k u)_{\lambda}
 =  \int_{\Omega} (\Pi_{j+1} G_f^*u)_{\lambda} \wedge w_{e,f},
 \quad (e,f)  \in \Delta_{j,m},
\]
where $0 \le m \le n$ and $-1 \le j < n-m$. We recall that for $e \in
\Delta_j(f^*)$, the functions $w_{e,f}$ are trimmed linear
$(n-j-1)$-forms with support in $\Omega_f$, and as a consequence, the
operator $Q_{e,f}^k$ maps $k$-forms to $(k -j -1)$-forms.  In
particular, if $k < j+1$, then $Q_{e,f}^k \equiv 0$.  Since $z_{e,f} =
(\delta^+ w)_{e,f}$, we also have
\begin{equation}
  \label{Requiv}
  (\delta^+ Q^k u)_{e,f} = R_{e,f}^k u.
\end{equation}
Since the operators $Q_{e,f}^k$ are constructed from a similar
procedure as the operators $R_{e,f}^k$, the new operators will also
preserve piecewise smoothness and the piecewise polynomial spaces. In
particular, we have for $(e,f) \in \Delta_{j,m}$ that
\begin{equation}\label{tilde-R-prop}
    b^{-(j+1)}Q_{e,f}^k (\P\Lambda^k(\T_f))
   \subset \P\Lambda^k(\S_f^c),
\end{equation}
where $\P$ can either be $\P_r$ or $\P_r^-$.  Here, the extra factor
$b^{-1}$, as compared to the mapping properties of the corresponding
operators $R_{e,f}^k$ given in Lemma~\ref{lem:R-pol-prop}, is due to
the fact that the order of the forms $w_{e,f}$ are reduced by one as
compared to the forms $z_{e,f}$. We refer to the proof of Proposition
5.2 of \cite{bubble-II} for further details.
 
\begin{lem}\label{tilde-R} Let $1 \le m \le n$ and  $j = n-m$.
Assume that $f \in \Delta_{m-1}(\T)$ and $g \in \bar \Delta(f)$. The identity 
\[
\sum_{e \in \Delta_{j}(f^*)} \phi_e \wedge L_g^*
\Big((-1)^{j+1}Q_{e,f}^{k+1}du - dQ_{e,f}^{k} u\Big) = 
\sum_{e \in \Delta_{j}(f^*)} \beta_e(f) \wedge  L_g^*
R_{e,f}^{k}u,
\]
holds on $\Omega_f$.
\end{lem}

\begin{proof}
For a fixed $e \in \Delta_j(f^*)$, we have
\begin{align*}
(-1)^{j+1}Q_{e,f}^{k+1}du - dQ_{e,f}^{k} u
  &= \int_{\Omega} \Big((-1)^{j+1}\Pi_{j +1}G_f^*du
  - d_{S} \Pi_{j +1}G_f^* u\Big)
  \wedge w_{e,f}\\
&= (-1)^{j+1}  \int_{\Omega} d_{\Omega} \Pi_{j}G_f^* u \wedge w_{e,f}
= \int_{\Omega}  \Pi_{j}G_f^* u \wedge dw_{e,f},
\end{align*}
where we have used the identity \eqref{dG*}, and the local support of
$w_{e,f}$.  However, from \eqref{F-def-special} we have that
\[
\sum_{e \in \Delta_{j}(f^*)} \phi_e \otimes
L_g^*\int_{\Omega}  \Pi_{j}G_f^* u \wedge dw_{e,f}
=\sum_{e \in \Delta_{j}(f^*)} \beta_e(f) \otimes  L_g^* R_{e,f}^{k}u
\]
on $\Omega_f$, and hence the desired identity follows from
\eqref{tensor-to-wedge}.
\end{proof}
 
\section{The  local operators $K_{m,f}^k$}
\label{sec:Kmgk}
We recall that the operators $\{B_f^k\}_{f \in \Delta}$, appearing in
the decomposition \eqref{main-decomp}, will be defined by the operator
$W^k$, mapping into the space of trimmed linear $k$-forms, and from
the local operators $\{K_{m,f}^k \}_{f \in \Delta_j}$, $j=m,m-1$,
by formulas \eqref{def-B}--\eqref{def-B-n}.  The purpose of this
section is to define the operators $\{K_{m,f}^k\}$.  For each value of
$m$, $1 \le m \le n-1$, we define the operators $K_{m,f}^k$, for $ f
\in \Delta_m$, by
\begin{equation}\label{def-Kmf}
   K_{m,f}^k u= \sum_{g \in \bar \Delta(f)}(-1)^{|f| -|g|}L_g^* A_f^k u.
\end{equation}
This is the obvious generalization of the corresponding operators
defined for $k=0$ in Section~\ref{sec:scalar} above.  From the
properties of the operators $A_f^k$, cf. Section~\ref{sec:average}, it
follows that the operator $L_g^* A_f^k$, and hence $K_{m,f}^k$, has
domain of dependence $\Omega_f$.  Furthermore, the operator
\eqref{def-Kmf} commutes with the exterior derivative, preserves
piecewise smoothness and the piecewise polynomial spaces, and by the
cancellation argument, cf. Section~\ref{sec:scalar}, we obtain that
the functions $K_{m,f}^k u$ have support on $\Omega_f$.

For $f \in \Delta_{m-1}$, the generalization of the operators
$K_{m,f}^0$, introduced in Section~\ref{sec:scalar}, to the case of
$k$-forms is less obvious. In this case, we define $K_{m,f}^k$ by an
alternating sum of the form
\begin{equation}
  \label{K-form}
 K_{m,f}^k = \sum_{g \in \bar \Delta(f)}(-1)^{|f| -|g|} K_{m,f,g}^k,
\end{equation}
 where  
\begin{multline}\label{def-K^k}
  K_{m,f,g}^k u =  - L_g^*A_f^k u + \sum_{j=0}^{n-m}
  \frac{(-1)^{j-1}}{\rho_g^{j+1}}
  \sum_{e  \in \Delta_j(f^*)} \Big(\phi_e + \psi_{e,g}(f)\Big)
 \wedge L_g^*R_{e,f}^k u\\
    + \sum_{\substack{e  \in \Delta_{n-m}(f^*)\\j = n-m}} (-1)^{j}
    \Big(d \frac{\phi_e}{\rho_g^{j+1}}\Big) \wedge L_g^* Q_{e,f}^ku.
 \end{multline}
These operators reduce to the corresponding operators $K_{m,f}^0$ if
$k=0$. To see this, observe that all the operators $Q_{e,f}^0$ are
identically zero. Similarly, $R_{e,f}^0 = 0$ if $(e,f) \in
\Delta_{j,m}$, $j > 0$, and if $j= 0$ then $R_{e,f}^0= -
A_{\<e,f\>}^0$. As a consequence, we obtain from \eqref{K-form} and
\eqref{def-K^k} that
\begin{multline*}
 K_{m,f}^0 u =  \sum_{g \in \bar \Delta(f)}(-1)^{|f| -|g|}\Big[- L_g^*A_f^0 u
 +   \rho_g^{-1}
  \sum_{i  \in I(f^*)} \Big(\lambda_i + \psi_{x_i,g}(f)\Big)
 \wedge L_g^*A_{\<x_i,f\>}^0 u \Big]\\
 = \sum_{g \in \bar \Delta(f)}(-1)^{|f| -|g|} \rho_g^{-1}
 \Big[\sum_{i  \in I(f^*)}\Big( \lambda_i - \frac{\rho_f}{|f^*|} \Big)
 \wedge L_g^*A_{\<x_i,f\>}^0 u \Big],
\end{multline*}
where we have used \eqref{def-zf} and \eqref{psi-spec}. Hence, we can
conclude that the definition above agrees with the definition given in
Section~\ref{sec:scalar} when $k=0$, cf. formula \eqref{def-Kmg0}.  We
note that the domain of dependence of all the operators in
\eqref{def-K^k} is $\Omega_f$ and hence the domain of dependence of
$K_{m,f}^k$ is $\Omega_f$.  Furthermore, it follows from
\eqref{def-psi} that
\[
\tr_{\lambda_i = 0} \psi_{e,g}(f) = \tr_{\lambda_i = 0} \psi_{e,\<x_i,g\>}(f),
\quad x_i \in f \cap g^*.
\]
By the cancellation argument introduced in Section~\ref{sec:scalar},
it is then easy to see that the functions $K_{m,f}^k u$ have support
on $\Omega_f$. In fact, if $g \in \bar \Delta(f)$ and $i \in I(f \cap
g^*)$, then
\[
\tr_{\lambda_i = 0} \Big(K_{m,f,\<x_i,g\>}^ku - K_{m,f,g}^ku\Big) = 0,
\]
which shows that $K_{m,f}^k u$ has support on
$\Omega_{x_i}$. Furthermore, since $i \in I(f)$ is arbitrary, the
support of $K_{m,f}^k u$ must be limited to
\[
\Omega_f = \bigcap_{i \in I(f)} \Omega_{x_i}.
\]
A key step to show that the operator $K_{m,f}^k$ commutes with the
exterior derivative and preserves piecewise smoothness is the
following alternative expression of the first part of the operator
$K_{m,f,g}^k$.

\begin{lem}\label{lem:K}
  Assume that $1 \le m \le n-1$. For each  $f \in \Delta_{m-1}(\T)$ and
  $g \in \bar \Delta(f)$, the identity 
\begin{multline}
\label{ident-K}
- Lg^* A_f^k u + \sum_{j=0}^{n-m}
  \frac{(-1)^{j-1}}{\rho_g^{j+1}}
  \sum_{e  \in \Delta_j(f^*)} \Big(\phi_e + \psi_{e,g}\Big)
  \wedge L_g^*R_{e,f}^k u
  \\
 = \sum_{j=1}^{n-m} \sum_{e \in \Delta_j(f^*)}\Big[d(\mu_e\wedge L_g^*b^{-j}R_{e,f}^k u)
   + \mu_e \wedge L_g^*b^{-j}R_{e,f}^{k+1} du\Big]
 - \sum_{\substack{e \in \Delta_j(f^*)\\ j=n-m}}
\frac{\beta_e}{\rho_g^{j+1}} \wedge L_g^*R_{e,f}^k u
\end{multline}
holds on $\Omega_f$.  Here, the functions $\mu_e,\, \beta_e,\,
\psi_{e,g}$ are all associated to the simplex $f$.
\end{lem}
The proof of the lemma above is partly technical.  Therefore, we delay
the proof, and we will first use the identity to prove the following
key result.

\begin{lem}\label{prop:K} Let $1 \le m \le n-1$ and assume
  $f \in \Delta_j(\T)$, $j=m,m-1$.  The function $K_{m,f}^ku$ has
   domain of dependence $\Omega_f$ and support on
  $\Omega_f$. Furthermore, the operator $K_{m,f}^k$ commutes with the
  exterior derivative and maps the spaces $\Lambda^k(\T_f)$,
  $\P_r^-\Lambda^k(\T_f)$, and $\P_r\Lambda^k(\T_f)$ to themselves.
\end{lem}

\begin{proof} 
Following the discussion above, it remains to show that the operator
defined by \eqref{K-form} and \eqref{def-K^k} preserves piecewise
smoothness and, in particular, the piecewise polynomial spaces, and
that this operator commutes with the exterior derivative. In fact, we
will show that each of the operators $K_{m,f,g}^k$ has these
properties, where $ g \in \bar \Delta(f)$. Furthermore, due to the
support property of the functions $K_{m,f}^ku$ already verified, it is
enough to establish these properties on the domain $\Omega_f$.

It is a consequence of the identity of Lemma~\ref{lem:K} that the
operator $K_{m,f,g}^k$ admits the alternative representation
\begin{multline*}
    K_{m,f,g}^k u
    = \sum_{j=1}^{n-m} \sum_{e \in \Delta_j(f^*)}
    \Big(d(\mu_e\wedge L_g^*b^{-j}R_{e,f}^k u)
  + \mu_e \wedge L_g^*b^{-j}R_{e,f}^{k+1} du\Big)\\
 +\sum_{\substack{e  \in \Delta_{j}(f^*)\\j =n-m}} \Big((-1)^{j}
    \Big(d \frac{\phi_e}{\rho_g^{j+1}}\Big)
    \wedge L_g^* Q_{e,f}^ku
    - \frac{\beta_e}{\rho_g^{j+1}} \wedge L_g^*R_{e,f}^k u\Big).
\end{multline*}
However, by Lemma~\ref{tilde-R}, the two last terms can be rewritten
in the form
\begin{multline*}
     \sum_{\substack{e  \in \Delta_{j}(f^*)\\j =n-m}} \Big[(-1)^{j}
     \Big(d \frac{\phi_e}{\rho_g^{j+1}}\Big)  \wedge L_g^* Q_{e,f}^ku
     -  \frac{\phi_e}{\rho_g^{j+1}} \wedge L_g^*\Big((-1)^{j+1}
   Q_{e,f}^{k+1}du - dQ_{e,f}^{k} u\Big)\Big]\\
   = \sum_{\substack{e  \in \Delta_{j}(f^*)\\j =n-m}} (-1)^{j}
   \Big(d\Big(\frac{\phi_e}{\rho_g^{j+1}} \wedge L_g^* Q_{e,f}^ku \Big)
   +\frac{\phi_e}{\rho_g^{j+1}} \wedge L_g^*  Q_{e,f}^{k+1} du \Big).
\end{multline*}
As a consequence, it follows that the operator $K_{m,f,g}^k$ can be
expressed as
\begin{multline*}
K_{m,f,g}^k u =  
\sum_{j=1}^{n-m} \sum_{e \in \Delta_j(f^*)}\Big(d(\mu_e(f)
\wedge L_g^*b^{-j}R_{e,f}^k u)
  + \mu_e(f) \wedge L_g^*b^{-j}R_{e,f}^{k+1} du\Big)\\
   + \sum_{\substack{e  \in \Delta_{j}(f^*)\\j =n-m}} (-1)^{j}
    \Big(d\Big(\phi_e   \wedge L_g^* b^{-(j+1)}Q_{e,f}^ku \Big)
   +\phi_e \wedge L_g^*b^{-(j+1)}  Q_{e,f}^{k+1} du\Big).
   \end{multline*}
From this representation, commuting with the exterior derivative is
obvious, and the space preserving properties follow from
\eqref{L-star-prop} and the space preserving properties of the
operators $R_{e,f}^k$ and $Q_{e,f}^k$ given in
Lemma~\ref{lem:R-pol-prop} and \eqref{tilde-R-prop}.  In particular,
for the trimmed piecewise polynomial spaces, we use the fact that
$\P_1^-\Lambda^j \wedge \P_{r-1}^-\Lambda^{k-j} \subset
\P_{r}^-\Lambda^{k}$, cf. \cite[Section 3.3]{acta}.
\end{proof}

Next, we will establish Lemma~\ref{lem:K}.
\begin{proof} (of Lemma~\ref{lem:K})
 From the identity \eqref{psi-id}, we obtain
 \begin{multline*}
\sum_{j=0}^{n-m}\sum_{e  \in \Delta_j(f^*)} \frac{(-1)^{j-1}}{\rho_g^{j+1}}  
\Big(\phi_e + \psi_{e,g}\Big) \wedge L_g^* R_{e,f}^k u
\\
= \sum_{j=0}^{n-m}\sum_{e  \in \Delta_j(f^*)}
\Big[d\Big(\frac{\mu_e}{\rho_g^j}\Big)
  - \frac{\beta_e}{\rho_g^{j+1}}\Big] \wedge L_g^*R_{e,f}^k u
\\
  = \sum_{\substack{e  \in \Delta_{j}(f^*)\\j= n-m}}\Big[d\Big(\frac{\mu_e}{\rho_g^j}\Big)
    - \frac{\beta_e}{\rho_g^{j+1}}\Big] \wedge L_g^*R_{e,f}^k u
  +  \sum_{j=0}^{n-m-1}    \sum_{e  \in \Delta_j(f^*)} 
  \Big[d\Big(\frac{\mu_e}{\rho_g^j}\Big)
    - \frac{(\partial \mu)_e}{\rho_g^{j+1}}\Big] \wedge L_g^*R_{e,f}^k u.
  \end{multline*}
where we have used the identity $\beta_e = (\partial \mu)_e$ from
Lemma~\ref{lem:mu-1}, and where $\partial = \partial(f^*)$. Recall
also that for $e \in \Delta_0(f^*)$, we obtain from \eqref{def-zf}
that
\[
  \sum_{e \in \Delta_0(f^*)}d\mu_e \wedge L_g^*R_{e,f}^k u  = L_g^*A_f^k u,
\]
since $d\mu_e = -1/|f^*|$ and $R_{e,f}^k u = - A_{\<e,f\>}^k$.  Hence,
the left hand side of \eqref{ident-K} is given by
\begin{multline}
\label{lhs-rewrite}
\sum_{\substack{e  \in \Delta_{j}(f^*)\\j= n-m}}\Big[d\Big(\frac{\mu_e}{\rho_g^j}\Big)
    - \frac{\beta_e}{\rho_g^{j+1}}\Big] \wedge L_g^*R_{e,f}^k u
-\sum_{e  \in \Delta_0(f^*)} \frac{(\partial \mu)_e}{\rho_g} \wedge L_g^*R_{e,f}^k u
\\
+ \sum_{j=1}^{n-m-1}    \sum_{e  \in \Delta_j(f^*)} 
  \Big[d\Big(\frac{\mu_e}{\rho_g^j}\Big)
    - \frac{(\partial \mu)_e}{\rho_g^{j+1}}\Big] \wedge L_g^*R_{e,f}^k u.
\end{multline}
On the other hand, it follows from \eqref{Rkd-delta} and the Leibniz
rule for the wedge product, that
\begin{multline*}
  d(\mu_e\wedge L_g^*b^{-j}R_{e,f}^k u)
  + \mu_e \wedge L_g^*b^{-j}R_{e,f}^{k+1} du
\\
=  \Big(d\frac{\mu_e}{\rho_g^j}\Big)\wedge L_g^*R_{e,f}^k u
+ \mu_e \wedge L_g^*b^{-j}((-1)^{j-1}dR_{e,f}^k u + R_{e,f}^{k+1} du)
\\
= \Big(d\frac{\mu_e}{\rho_g^j}\Big)\wedge L_g^*R_{e,f}^k u
- \mu_e \wedge L_g^*b^{-j}(\delta R^ku)_{e,f},
\end{multline*}
for $e \in \Delta_j(f^*)$, $0 \le j \le n-m$.  As a consequence, the
right hand side of \eqref{ident-K} is given by
\begin{multline*}
= \sum_{j=1}^{n-m-1} \sum_{e \in \Delta_j(f^*)}
\Big(d\frac{\mu_e}{\rho_g^j}\Big)\wedge L_g^*R_{e,f}^k u
+ \sum_{\substack{e \in \Delta_j(f^*)\\j= n-m}}\Big[\Big(d\frac{\mu_e}{\rho_g^{j}}\Big)
  - \frac{\beta_e}{\rho_g^{j+1}}\Big] \wedge L_g^*R_{e,f}^k u
\\
- \sum_{j=0}^{n-m-1} \sum_{e \in \Delta_j(f^*)}
\partial \mu_e \wedge L_g^*b^{-(j+1)}( R^ku)_{e,f},
\end{multline*}
where we have used \eqref{sum-by-parts} to rewrite the last term.  
 If we combine terms and compare this with
  \eqref{lhs-rewrite}, we obtain \eqref{ident-K}.
\end{proof}

The definitions of the operators $\{K_{m,f}^k\}$, given above,
combined with the operators $\{W^k\}$ introduced in
Section~\ref{sec:outline}, complete the construction of all the
operators required for the decomposition \eqref{main-decomp},
cf. \eqref{def-B} and \eqref{def-B-n}.
 
\begin{prop}\label{prop:B} 
 The operators $\{B_f^k\}$, defined by \eqref{def-B} and
 \eqref{def-B-n}, have domain of dependence $\Omega_f$ if $f \in
 \Delta_m$, $ m<n$, and $\Omega_f^E$ if $m=n$. Furthermore, the
 functions $\B_f^ku$ have support in $\Omega_f$.
\end{prop}
 
\begin{proof}
All these properties follow directly from the corresponding properties
of the operators $\{K_{m,f}\}$ given in Lemma~\ref{prop:K}, except for
the domain of dependence property in the case $f \in
\Delta_n$. However, it is a consequence \eqref{def-B-n} that in this
case the operator $B_f^k$ has a domain of dependence included in
\[
\big(\bigcup_{e \in \Delta_k(f)} \Omega_e^E \big) \cup
\big(\bigcup_{g \in \Delta(f)}\Omega_g\big) \subset \Omega_f^E.
\]
\end{proof}

To complete the proof of the main theorem of the paper,
Theorem~\ref{thm:main}, it remains to verify the identity
\eqref{Cmk-ident} and to establish the desired operator bounds. This
will be done in the two next sections.

\section{Verifying the fundamental identity}
\label{sec:fund-ident}
The purpose of this section is to establish the fundamental identity
\eqref{Cmk-ident}, i.e.,
\[
  C_m^ku- \sum_{\substack{f \in \Delta_j\\j= m,m-1}} K_{m,f}^ku = C_{m-1}^k u,
   \quad 1 \le m \le n-1.
\]
Therefore, we have to study sums of the operators $K_{m,f}^k$
introduced in the previous section.  As a preliminary step, we study
sums of expressions corresponding to a part of the definition
\eqref{def-K^k}.
  
\begin{lem} Assume that $1 \le m \le n-1$. Then
\begin{multline}\label{Cmk-ident-pre1}
   \sum_{\substack{(e,f)  \in \Delta_{j.m-1}\\0 \le j \le n-m}}(-1)^{j-1}
    \sum_{g \in \bar \Delta(f)}(-1)^{|f| -|g|}
  \frac{\psi_{e,g}(f)}{\rho_g}  \wedge L_g^*b^{-j}R_{e,f}^k u\\
 = - \sum_{\substack{(e,f)  \in \Delta_{j.m-2}\\0 \le j \le n-m}}(-1)^{j-1} 
 \sum_{g \in \bar \Delta(f)} (-1)^{|f| -|g|}
  \frac{\phi_e}{\rho_g} \wedge L_g^* b^{-j} R_{e,f}^k u.
\end{multline}
\end{lem}

\begin{proof}
Recall from \eqref{def-psi} that $\psi_{e,g}(f) = 0$ if $g=
f$. Therefore, by changing the order of summation and then using
\eqref{residual-idRg}, the left hand side of \eqref{Cmk-ident-pre1}
can be rewritten as
\begin{multline*}
 \sum_{\substack{g \in \Delta_s \\ -1 \le s \le m-2}} (-1)^{m -|g|}
  \sum_{j=0}^{n-m}   \frac{(-1)^{j-1}}{\rho_g^{j+1}}
  \sum_{\substack{(e,f) \in \Delta_{j,m-1} \\ f \supset g}}
 \psi_{e,g}(f)  \wedge L_g^*R_{e,f}^k u
  \\
  = \sum_{\substack{g \in \Delta_s \\ -1 \le s \le m-2}} (-1)^{m -|g|}
  \sum_{j=0}^{n-m}   (-1)^{j-1}
  \sum_{\substack{(e,f) \in \Delta_{j,m-2} \\ f \supset g}}
  \frac{\phi_e}{\rho_g} \wedge L_g^* b^{-j} R_{e,f}^k u,
\end{multline*}
where the right hand side corresponds exactly to the right hand side
of \eqref{Cmk-ident-pre1}.
\end{proof}

Next, we consider a corresponding sum of the last term of the
definition \eqref{def-K^k}.

\begin{lem} Assume that $1 \le m \le n-1$.  Then
\begin{multline}\label{Cmk-ident-pre2}
    \sum_{\substack{(e,f) \in \Delta_{j,m-1}\\j=n-m}}
   \sum_{g \in \bar \Delta(f)} (-1)^{|f|-|g|}
  \Big(d \frac{\phi_e}{\rho_g^{j+1}}\Big)
     \wedge L_g^*Q^k_{e,f} u
\\
     =-  \sum_{\substack{(e,f) \in \Delta_{j,m-2}\\ j = n-m+1}}  
     \sum_{g \in \bar \Delta(f)} (-1)^{|f|-|g|}
     \frac{\phi_e}{\rho_g} \wedge   L_g^* b^{-j} R_{e,f}^k u.
\end{multline}
\end{lem}
\begin{proof}
Let $f \in \Delta_{m-1}$ and $g \in \bar \Delta(f)$ be fixed.  Since
$\Omega_f \subset \Omega_g$, it follows that $\rho_g = \sum_{i \in
  I(g^*)} \lambda_i$ on $\Omega_f$.  Since $\dim f^* = n-m$ and $e \in
\Delta_{n-m}(f^*)$, $\phi_{[x_i,e]} =0$ unless $i \in I(f)$ and so we
have
  \[
 d\Big(\frac{\phi_e}{\rho_g^{n-m+1}}\Big) = \frac{1}{\rho_g^{n-m+2}}
 \sum_{i \in I(f \cap g^*)}\phi_{[x_i,e]},\quad \text{on } \Omega_f.
 \] 
As a consequence, when we restrict to $\Omega_f$, we can conclude that
\begin{equation*}
\sum_{\substack{e  \in \Delta_{j} (f^*)\\ j= n-m}}
\Big(d\frac{\phi_e}{\rho_g^{j+1}}\Big) \wedge L_g^* Q_{e,f}^k u
=   \sum_{i \in I(f \cap g^*)}
\sum_{\substack{e  \in \Delta_{j} (f(\hat x_i)^*)\\ e \supset x_i\\j =n-m+1}} 
 (-1)^{\sigma_{e(x_i)}}\frac{\phi_{e} }{\rho_g} \wedge
 L_g^* b^{-j}  Q_{e(\hat x_i),f}^k u.
\end{equation*}
Next, if we sum over all $g \in \bar \Delta(f)$ we obtain
\begin{multline}\label{Cmk-ident-pre3}
  \sum_{g \in \bar \Delta(f)} (-1)^{|f|-|g|}
 \sum_{\substack{e  \in \Delta_{j} (f^*)\\ j= n-m}}
\Big(d\frac{\phi_e}{\rho_g^{j+1}}\Big) \wedge L_g^* Q_{e,f}^k u\\
= \sum_{g \in \bar \Delta(f)} (-1)^{|f|-|g|}  \sum_{i \in I(f \cap g^*)}
\sum_{\substack{e  \in \Delta_{j} (f(\hat x_i)^*)\\ e \supset x_i\\j =n-m+1}} 
 (-1)^{\sigma_{e(x_i)}}\frac{\phi_{e} }{\rho_g} \wedge
 L_g^* b^{-j}  Q_{e(\hat x_i),f}^k u.
\end{multline}
This identity obviously holds on $\Omega_f$, and by the cancellation
argument, it also holds on $\Omega \setminus \Omega_f$, since both
sides of the identity vanish there. Furthermore, if we sum
\eqref{Cmk-ident-pre3} over all $f \in \Delta_{m-1}$, and use the fact
that $f = \<x_i,f(\hat x_i)\>$ for $i \in I(f)$, we obtain that the
left hand side of \eqref{Cmk-ident-pre2} can be expressed as
\begin{align*}
&\sum_{\substack{g \in  \Delta_s\\-1 \le s \le m-2}}\hskip -8pt (-1)^{m-|g|}
\sum_{\substack{f \in \Delta_{m-1}\\ f \supset g}}
\sum_{i \in I(f \cap g^*)}
\sum_{\substack{e  \in \Delta_{j}(f(\hat x_i)^*)\\ e \supset x_i\\j =n-m+1}} 
(-1)^{\sigma_{e(x_i)}}\frac{\phi_{e} }{\rho_g}
\wedge  L_g^* b^{-j}  Q_{e(\hat x_i),f}^k u\\
 &=  \sum_{\substack{g \in  \Delta_s\\-1 \le s \le m-2}}\hskip -8pt (-1)^{m-|g|}
 \sum_{\substack{f \in \Delta_{m-2}\\ f \supset g}}
 \sum_{\substack{e  \in \Delta_{j} (f^*)\\j =n-m+1}} 
 \frac{\phi_{e} }{\rho_g}  \wedge  L_g^* b^{-j}
\sum_{i \in I(e)}(-1)^{\sigma_e(x_i)}Q_{e(\hat x_i),\<x_i,f\>}^k u\\
&=- \sum_{\substack{(e,f) \in \Delta_{j,m-2}\\ j = n-m+1}}  
     \sum_{g \in \bar \Delta(f)} (-1)^{|f|-|g|}
     \frac{\phi_e}{\rho_g} \wedge   L_g^* b^{-j} (\delta^+ Q^k u)_{e,f},
      \end{align*}
 where we have used the fact that for $g$ fixed, we have 
 \[
 \sum_{\substack{f \in \Delta_{m-1}\\ f \supset g}}
 \sum_{i \in I(f \cap g^*)}\sum_{\substack{e  \in \Delta_{j} (f(\hat x_i)^*)\\ e \supset x_i}}
 = \sum_{\substack{f \in \Delta_{m-2}\\ f \supset g}} \sum_{e  \in \Delta_{j} (f^*)}
 \sum_{i \in I(e)}.
  \]
However, by \eqref{Requiv}, the final term above is exactly the right
hand side of \eqref{Cmk-ident-pre2}.
\end{proof}

The main result of this section now follows from the two lemmas above.

\begin{prop} \label{prop:Cmk-ident}
Let $1 \le m \le n-1$. Then the identity \eqref{Cmk-ident} holds.
\end{prop}

\begin{proof}
 It follows from \eqref{Cm-rewritten} and
\eqref{def-Kmf} that
\[
C_m^k u - \sum_{f \in \Delta_m} K_{m,f}^k u=
\sum_{\substack{(e,f) \in \Delta_{j,m-1}\\ 0 \le j \le n-m}}(-1)^{j-1} 
\sum_{g \in \bar \Delta(f)} (-1)^{|f|-|g|}
\frac{\phi_e}{\rho_g} \wedge L_{g}^* b^{-j} R_{e,f}^ku.
\]
On the other hand, it follows from \eqref{K-form}, \eqref{def-K^k},
and the two identities \eqref{Cmk-ident-pre1} and
\eqref{Cmk-ident-pre2} derived above, that
\begin{multline*}
  \sum_{f \in \Delta_{m-1}} K_{m,f}^k u = - \sum_{f \in \Delta_{m-1}}
  \sum_{g \in \bar \Delta(f)} (-1)^{|f|-|g|}
  L_g^*A_f^k u
    \\
 +   \sum_{\substack{(e,f) \in \Delta_{j,m-1}\\ 0 \le j \le n-m}}(-1)^{j-1} 
    \sum_{g \in \bar \Delta(f)} (-1)^{|f|-|g|}
\frac{\phi_e}{\rho_g} \wedge L_{g}^* b^{-j} R_{e,f}^ku.
  \\
- \sum_{\substack{(e,f) \in \Delta_{j,m-2}\\0 \le j \le n-m+1}} (-1)^{j-1}
 \sum_{g \in \bar \Delta(f)} (-1)^{|f|-|g|}
\frac{\phi_e}{\rho_g}  
 \wedge L_g^*b^{-j}R_{e,f}^k u.
\end{multline*}
Therefore, by comparing the two formulas above, we obtain that 
\begin{multline*}
  C_m^k u - \sum_{\substack{f \in \Delta_j\\ j = m,m-1}} K_{m,f}^k u
  = \sum_{f \in \Delta_{m-1}} 
\sum_{g \in \bar \Delta(f)} (-1)^{|f|-|g|}L_g^*A_f^k u\\
+ \sum_{\substack{(e,f) \in \Delta_{j,m-2}\\0 \le j \le n-m+1}} (-1)^{j-1}
 \sum_{g \in \bar \Delta(f)} (-1)^{|f|-|g|}
\frac{\phi_e}{\rho_g}  
 \wedge L_g^*b^{-j}R_{e,f}^k u,
\end{multline*}
and by \eqref{Cm-rewritten}, the right hand side is exactly $C_{m-1}^k
u$. This completes the proof.
\end{proof}

\section{Bounding the operator norms}
\label{sec:bounds}
To complete the proof of Theorem~\ref{thm:main}, it remains to show
that all the operators, $W^k$ and $B_f^k$, of the decomposition
\eqref{main-decomp} are bounded in $L^2\Lambda^k(\Omega)$, and satisfy
a stable decomposition property.  This will be achieved by
Proposition~\ref{prop:L2bound} below. In fact, since these operators
commute with the exterior derivative, they will also be bounded on the
Sobolev space $H\Lambda^k(\Omega)$.

The various constants that appear in the bounds below will depend on
the mesh $\T$ through the shape-regularity constant $c_{\T}$, defined
by
\begin{equation*}
  c_{\T} = \max_{T \in \Delta_n(\T)} \frac{\diam(T)}{\diam(\Ball_{T})},
\end{equation*}
where $\Ball_{T}$ is the largest ball contained in $T$.  The
consequence of this is that if we consider a family of meshes, $\{\T^h
\}$, parametrized by a real parameter $h \in (0,1]$, typically
  obtained by mesh refinements, the bounds will be uniform with
  respect to $h$ as long as we restrict to a family with a uniform
  bound on the constants $\{c_{\T^h} \}$.  In addition to the
  dependence explicitly stated, the constants in bounds below will
  also depend on the space dimension $n$.  Throughout this section, we
  will assume that the operators under investigation are applied to
  piecewise smooth differential forms.  However, since the space
  $\Lambda^k(\T)$ is dense in $L^2\Lambda^k(\Omega)$, it a consequence
  of the domain of dependence result in Proposition~\ref{prop:B} and
  the bound obtained in Proposition~\ref {prop:L2bound} below, that
  all the operators $B_f^k$, where $f \in \Delta_m$, can be extended
  to bounded operators from $L^2\Lambda^k(\Omega_f)$ to itself if $0
  \le m< n$, and from $L^2\Lambda^k(\Omega_f^E)$ to
  $L^2\Lambda^k(\Omega_f)$ when $m=n$.  To bound the norms of the
  operators comprising the new decomposition of the bubble transform
  developed in this paper, we will basically follow the approach
  developed in \cite[Section 8]{bubble-II}.  We recall that the
  decomposition \eqref{main-decomp} takes the form
\begin{equation*}
  u = W^ku + \sum_{m=0}^n \sum_{f \in \Delta_m} B_{f}^k u,
\end{equation*}
The main result of this section is the following bound. 

\begin{prop}
\label{prop:L2bound}
There exists a constant $c$, depending on the shape-regularity
constant $c_{\T}$, such that for $0 \le k \le n$, we have
\[
\|W^k u\|_{L^2(\Omega)},  \Big(\sum_{f \in \Delta(\T)} 
\|B_f^k u\|_{L^2(\Omega_f)}^2 \Big)^{1/2}  \le c \|u\|_{L^2(\Omega)}.
\]
\end{prop}

To establish the bounds in Proposition~\ref{prop:L2bound}, we will
need some preliminary results.  We define the overlap of a set of
subdomains as the smallest upper bound for the number of domains which
will contain any fixed element $T \in \Delta_n$.  The overlap of the
set of macroelements, $\{ \Omega_f \}_{f \in \Delta_m}$, will only
depend on $m$ and the space dimension $n$, while the overlap for the
set of extended macroelements, $\{\Omega_{f}^{E} \}$, will depend on
the mesh $\T$. Another important property of the extended
macroelements is the variation of the size of the elements.  We define
$h_f = \max_{T \in \Delta_n(\T_f^E)} \diam(T)$, where $\T_f^E$ is the
restriction of the mesh $\T$ to $\Omega_f^E$.  The following result,
established in Lemmas 8.2 and 8.3 of \cite{bubble-II}, shows that
these domains allow bounded overlap and local quasi-uniformity in the
following sense.

\begin{lem}\label{lem:overlap}
There is a constant $c$, depending on $\T$ only through the
shape-regularity constant $c_{\T}$, which bounds the overlap of the
domains $\{\Omega_f^E \}_{f \in \Delta(\T)}$. Furthermore,
\begin{equation}\label{max-min}
h_f \le
c \min_{T \in \Delta_n(\T_f^E)} \diam(T), 
\quad f \in \Delta(\T).
\end{equation}
\end{lem}
  
From the definitions of the operators $K_{m,f}^k$, given by
\eqref{def-Kmf}--\eqref{def-K^k}, we will need appropriate bounds for
the functions $\phi_e$, $\psi_{e,g}(f)$, and also for the functions
$w_{e,f}$ and $z_{e,f}$ which are used to define the order reduction
operators $Q_{e,f}^k$ and $R_{e,f}^k$.  All these functions are
trimmed linear forms with local support.  In particular, if $(e,f) \in
\Delta_{j,m}$ and $g \in \bar \Delta(f)$, then $\psi_{e,g}(f)$ belongs
to $\P_1^-\Lambda^j(\T,g^*)$, $w_{e,f} \in
\dzero\P_1^-\Lambda^{n-j-1}(\T_f)$, and $z_{e,f} \in
\dzero\P_1^-\Lambda^{n-j}(\T_f) \cap
\dzero\P_1^-\Lambda^{n-j}(\T_e^E)$.  In general, if $w$ is any trimmed
linear form, say $w \in \P_1^-\Lambda^j(\T)$, then $w$ admits a unique
expansion of the form
\[
w = \sum_{e \in \Delta_j} c_e \phi_e,
\]
where $\{c_e\}$ are real coefficients. If $\max_{e \in \Delta_j}
|c_e|$ can be bounded by a quantity which only depends on the mesh
$\T$ through the mesh regularity constant, we will say that {\it $w$
  admits a uniformly bounded expansion.}  It is a consequence of the
bound \eqref{max-min} that for $g \in \bar \Delta(f)$ and $e \in
\Delta_j(f^*)$ we have
\begin{equation}\label{phi-bound}
 \|\phi_e/\rho_g\|_{L^\infty(\Omega)} \le c h_e^{-j},
\end{equation}
where $c$ depends on the shape-regularity constant. Note, in
particular, that $g = \emptyset$ gives a bound on the $L^\infty$-norm
of $\phi_e$. Next, recall that the coefficients $\{a_{e,e^\prime}\} =
\{a_{e,e^\prime}(f)\}$ of functions $\{\mu_e(f)\}$, given by
\eqref{mu-expand}, can be computed recursively with respect to
increasing values of $j$ by the algebraic systems
\eqref{mu-rel-algebraic}, \eqref{mu-extra}. There are no explicit mesh
dependent quantities present in the systems \eqref{mu-rel-algebraic},
\eqref{mu-extra}.  Only the number of equations depends on the mesh
through the number of elements in $\Delta_j(f^*)$, and this number can
be bounded by the shape-regularity constant. Therefore, since
$a_{e,\emptyset}(f) = - 1/|f^*|$ for $e \in \Delta_0(f^*)$, we can
conclude that all functions $\{\mu_e(f)\}$ admit uniformly bounded
expansions.  Furthermore, since the functions
$\{\tr_{\Omega_f}\beta_e(f)\}$ and $\{\psi_{e,g}(f)\}$ are explicitly
defined from $\{ \mu_e(f) \}$, through \eqref{def-beta-alt} and
\eqref{def-psi}, the same conclusion holds for these function classes.
Hence, it follows from \eqref{mu-expand} \eqref{def-psi}, and
\eqref{phi-bound} that the forms $\mu_e \in \P_1^- \Lambda^{j-1}(\T,
f^*)$ and $\psi_{e,g}(f) \in \P_1^- \Lambda^{j}(\T, g^*)$ satisfy
\begin{equation}\label{psi-bound}
    \|\mu_e(f)\|_{L^\infty(\Omega)} \le c h_e^{-j+1}, \qquad
      \|\psi_{e,g}(f)/\rho_g\|_{L^\infty(\Omega)} \le c h_e^{-j},
\end{equation}
where $(e,f) \in \Delta_{j,m}$ and $ g \in \bar \Delta(f)$ for $0 \le
m \le n-1$, $0 \le j <n-m$.
 
The following lemma below is a key ingredient to establish
Proposition~\ref{prop:L2bound}.
\begin{lem}\label{lem:expansion}
The trimmed linear differential forms $w_{e,f}$ and
  $z_{e,f}$ admit uniformly bounded expansions.
\end{lem}
We will delay the proof of this result to the end of the section.
However, from this bound, combined with \eqref{phi-bound} and
Lemma~\ref{lem:overlap}, we immediately obtain the estimates
\begin{equation}\label{z-bound}
  h_f^{-1} \| w_{e,f} \|_{L^\infty(\Omega)},
  \| z_{e,f} \|_{L^\infty(\Omega)} \le c h_f^{j-n}, \quad (e,f) \in
\Delta_{j,m},
\end{equation}
where the constant $c$ depends on the shape regularity constant.

\begin{lem}\label{lem:L2-W}
The operator $W^k$ maps $L^2(\Omega)$ to itself, and with an operator
norm bounded by the shape-regularity constant.
\end{lem}

\begin{proof}
We recall that the operator $W^k$ is given by 
\[
 W^ku =   (-1)^{k-1}
\sum_{e \in \Delta_{k}}
   \phi_e \Big(\int_{\Omega} u \wedge z_{e,\emptyset}\Big),
\]
Since the function $z_{e,\emptyset}$ is supported on
$\Omega_e^E$ we obtain from \eqref{z-bound} that
\begin{equation*}
  \int_{\Omega} u \wedge z_{e,\emptyset}
  \le c h_e^{n/2} \| z_{e, \emptyset} \|_{L^\infty(\Omega_e^E)}\| u \|_{L^2(\Omega_e^E)}
\le c h_e^{k- n/2} \| u \|_{L^2(\Omega_e^E)},
\end{equation*}
where here, and below, the constant $c$ depends on the
shape-regularity constant, but is not necessarily the same at each
occurrence.  Furthermore, since the function $\phi_e$ has support on
$\Omega_e$, we obtain from \eqref{phi-bound} that $|\int_{\Omega}
\phi_e^2| \le c h^{n-2k}.$ Finally, the fact that $\phi_e = \phi_e
\kappa_e$, where $\kappa_e$ is the characteristic function of
$\Omega_e$, implies that
\[
\Big(\sum_{e \in \Delta_{k}}
   \phi_e \big(\int_{\Omega} u \wedge z_{e,\emptyset}\big)\Big)^2 \le 
   \sum_{e \in \Delta_{k}} \phi_e^2
   \big(\int_{\Omega} u \wedge z_{e,\emptyset}\big)^2 \sum_{e \in \Delta_{k}} \kappa_e.
   \]
Putting this together, we obtain 
\begin{multline*}
  \| W^k u \|_{L^2(\Omega)}^2 \le c \int_{\Omega}\Big(\sum_{e \in \Delta_{k}}
   \phi_e \big(\int_{\Omega} u \wedge z_{e,\emptyset}\big)\Big)^2 \\ 
   \le c \big(\sum_{e \in \Delta_{k}}\| u \|_{L^2(\Omega_e^E)}^2\big)
  \big(\| \sum_{e \in \Delta_{k}} \kappa_e \|_{L^{\infty}(\Omega)}\big)
 \le c \| u \|_{L^2(\Omega)}^2,
\end{multline*}
where the final inequality follows from the overlap properties of the
domains $\{\Omega_e\}$ and $\{\Omega_e^E\}$,
cf. Lemma~\ref{lem:overlap}.  This completes the proof.
\end{proof} 

In addition to the $L^2$-bound for the operator $W^k$, we will need
corresponding bounds for all the operators $K_{m,f}^k$.  We observe
from the definitions \eqref{def-Kmf}, \eqref{K-form}, and
\eqref{def-K^k} of these operators that we will also need appropriate
bounds for operators of the form $A_f^k$, $R_{e,f}^k$, and
$Q_{e,f}^k$, composed with the pullback $L_g^*$ for $g \in \bar
\Delta(f)$. In fact, the three operators $A, Q, R$ are all of the same
form. In general, let $f \in \Delta_m$, and assume that $w$ is a fixed
function in $\0 \Lambda^{n-j}(\T_f)$.  Consider the corresponding
operator, $\Q_j^k(w) : \Lambda^k(\T_f) \to \Lambda^{k-j}(\S_f^c)$,
given by
\[
\Q_j^k(w)u = \int_{\Omega} \Pi_j G_f^*u \wedge w.
\]
By following the steps of the derivation of the bound (8.11) of
\cite{bubble-II}, but where we retain the $L^\infty$-norm of $w$ instead
of replacing it by an upper bound, we obtain the following result.

\begin{lem}\label{Q-bound}
Assume that $f \in \Delta_m(\T)$, $0 \le m \le n-1$, and that $0 \le j
\le k$. If $w \in \0 \Lambda^{n-j}(\T_f)$ then the bound
\[
\| L_g^*b^{-j} \Q_j^k(w) u \|_{L^2(\Omega_f)} \le c h_f^n \|w\|_{L^\infty(\Omega_f)}
\| u \|_{L^2(\Omega_f)},
\]
holds for any $g \in \bar \Delta(f)$, where the constant $c$ only
depends on $\T$ through the shape-regularity constant $c_{\T}$.
\end{lem}

With the help of the results obtained above, the proof of
Proposition~\ref{prop:L2bound} is straightforward.

\begin{proof}(of Proposition~\ref{prop:L2bound})
Consider a typical term in the definition \eqref{def-K^k} of the
operator $K_{m,f}^k$ for $f \in \Delta_{m-1}$ given by
\begin{equation*}
\Big(\frac{\phi_e + \psi_{e,g}(f)}{\rho_g}\Big)
 \wedge L_g^*b^{-j}R_{e,f}^k u,
 \end{equation*}
where $1 \le m \le n-1$, $0 \le j \le \min(n-m,k)$, $e \in
\Delta_j(f^*)$, and $ g \in \bar \Delta(f)$. Since the operator
$R_{e,f}^k$ can be identified as $\Q_j^k(z_{e,f})$, it follows from
\eqref{phi-bound}, \eqref{psi-bound}, \eqref{z-bound} and
Lemma~\ref{Q-bound} that the $L^2(\Omega_f)$-norm of this term can be
bounded by
\[
\|\Big(\frac{\phi_e + \psi_{e,g}(f)}{\rho_g}\Big)\|_{L^\infty(\Omega_f)} \cdot 
\|L_g^*b^{-j}R_{e,f}^k u\|_{L^2(\Omega_f)} 
\le c  \| u \|_{L^2(\Omega_f)},
\]
when $g \in \bar \Delta(f)$.  For each $f \in \Delta$, there
are a finite number of such terms in the definition of the operator
$K_{m,f}^k$ and it  therefore follows that
\[
\| K_{m,f}^k u \|_{L^2(\Omega_f)} \le c \| u \|_{L^2(\Omega_f)}.
\]
From this bound and the finite overlap property of the domains $\{\Omega_f\}$,
we then obtain 
\[
\sum_{m=0}^{n-1} \sum_{\substack{f \in \Delta_s\\s = m,m-1}}
\| K_{m,f}^k u \|_{L^2(\Omega_f)}^2
\le c \sum_{m=0}^{n-1} \sum_{\substack{f \in \Delta_s\\s = m,m-1}}
\|  u \|_{L^2(\Omega_f)}^2
\le c \| u \|_{L^2(\Omega)}^2.
\]
However, as a consequence of Lemma~\ref{lem:L2-W} and
\eqref{def-B}--\eqref{def-B-n}, this bound implies the desired bound
on $\sum_f B_f^k$.
\end{proof}

Finally, the proof below will complete the discussion of this section.
\begin{proof}(of Lemma~\ref{lem:expansion})
Recall that the functions $\{w_{e,f} \}$ and $\{z_{e,f}\}$ are defined
inductively with respect to decreasing values of $m$ through the
relations \eqref{w-to-z}, \eqref{F-def}, and \eqref{F-def-special}. We
recall that $w_{\emptyset,f} = - (\kappa_f/\Omega_f)\vol$ for $f \in
\Delta_n$, corresponding to the case $m=n$. As an induction
hypothesis, we assume that all the functions $\{w_{e,f}\}$, for $(e,f)
\in \Delta_{j,m}$, $-1 \le j <n-m$, admit a uniformly bounded
expansion. As a consequence of \eqref{w-to-z}, we then have that the
same property holds for all $\{z_{e,f}\}$ for $(e,f) \in
\Delta_{j,m-1}$.  Furthermore, by expanding the functions $\mu_e(f)$,
we derive from \eqref{w-to-z} and \eqref{F-def} that for $(e,f) \in
\Delta_{j,m-1}$, $-1 \le j <n-m$,
\begin{equation}\label{F-def-alt}
w_{e,f} = (-1)^j \sum_{e^\prime \in \Delta_{j+1}(f^*)} a_{e^\prime,e} z_{e^\prime,f},
\end{equation}
where the coefficients $a_{e^\prime,e}$ are obtained from
$\mu_{e^\prime}(f)$, cf. \eqref{mu-expand}, i.e., $\mu_{e^\prime} =
\sum_e a_{e^\prime,e} \phi_e$.  Hence, we can conclude by the
induction hypothesis and the fact that $\mu_e^{\prime}$ has a
uniformly bounded expansion, that the left hand side of
\eqref{F-def-alt} also has a uniformly bounded expansion.

It remains to bound $w_{e,f}$ for $(e,f) \in \Delta_{n-m,m-1}$.  It
follows from \eqref{beta-expand} that the expansion of  $\beta_e =
\beta_e(f)$ is given from the corresponding expansion for $\mu_e(f)$ by 
\[
\beta_e = \sum_{e^\prime \in \Delta_j(f^*)} b_{e,e^\prime} \phi_{e^\prime},
\quad b_{e,\e^\prime} = (\delta a_{e,\cdot})_{e^\prime} + (-1)^j  1_{e,e^\prime},
\]
where $j = n-m$, and we obtain from \eqref{F-def-special} that 
\begin{equation}\label{F-def-special-alt}
  dw_{e,f} =   \sum_{e^\prime \in \Delta_j(f^*)} b_{e^\prime, e} z_{e^\prime,f},
  \quad j= n-m.
\end{equation}
As above, we already know that the right hand side of
\eqref{F-def-special-alt} admits a uniformly bounded expansion.
Furthermore, since $w_{e,f} \in \0\P_1^-\Lambda^{m-1}(\T)$, it can be
expanded in the form
\[
w_{e,f} = \sum_{\substack{g \in \Delta_{m-1}\\ g \supset f}} c_g \phi_g \,
\Longrightarrow dw_{e,f} =
\sum_{\substack{g \in \Delta_{m}\\ g \supset f}} (\delta_{m-1} c)_g \phi_g, 
\]
where we have used \eqref{span-d}. As a consequence , the coefficients
$\delta_{m-1}c$ are uniformly bounded.  Also by definition, $d_{ f}^*
w_{e,f} = 0$, which means that $\partial_{m-1} c= 0$, cf. \eqref{df*}.
Recall that by local exactness, the coefficients $\{c_g\}$ are
uniquely determined by $\partial c$ and $\delta c$.  Furthermore,
there are no mesh dependent quantities present in the matrix
representation of the operators $\delta$ and $\partial$, just $1, -1,
0$. Therefore, since the number of simplices in $\T_f$ is bounded by
the shape-regularity constant, we can conclude that
\[
\max_{\substack{g \in \Delta_{m-1}\\ g \supset f}} |c_g| \le c 
\max_{\substack{g \in \Delta_m\\ g \supset f}} (\delta c)_g,
\]
where also the constant $c$ is bounded by the shape-regularity
constant. This completes the induction step and hence the proof of the
Lemma.
\end{proof}

\section{Final remarks}
\label{sec:final}
As we have already mentioned in the introduction of the paper, if we
restrict to scalar valued piecewise polynomial functions, or
zero-forms, in two or three space dimensions, then the decompositions
studied here, cf. \eqref{main-decomp}, correspond to decompositions
constructed in \cite{additive-schwarz}.  However, the construction
done in this paper deviates from the one presented in
\cite{additive-schwarz}.  In both these procedures, the given function
is decomposed into the sum of a global piecewise linear function and
local bubbles with support on the macroelements associated to the
vertices, edges, faces and tetrahedrons of the mesh.  In
\cite{additive-schwarz}, it is established that the decomposition is
stable in $H^1$, cf. Theorem 2.1 of that paper, and by combining this
with the bounded overlap property of the macroelements, it follows
that the corresponding additive Schwarz preconditioner behaves
uniformly with respect to $hp$-refinements, i.e., in particular with
respect to the polynomial degree. Here we have established that the
decomposition \eqref{main-decomp} is uniformly stable in $L^2$,
cf. Proposition~\ref{prop:L2bound}, and, since the operators $B_f^k$
commute with the exterior derivative, the decomposition is also stable
in $H^1$. Furthermore, the decomposition \eqref{main-decomp}, and the
bounds of Proposition~\ref{prop:L2bound} are established for all
$k$-forms and any space dimension, and due to the invariant property
\eqref{space-preserve}, these bounds apply to all finite element
spaces of the form $\P_r\Lambda^k$ and $\P_r^-\Lambda^k$.  As a
consequence, by following the setup discussed for example in
\cite[Section 10]{acta}, we can use these results to design
preconditioners for the mixed systems corresponding to discrete
Hodge-Laplace problems, which will lead to condition numbers that are
bounded uniformly with respect to $hp$-refinements. In particular, in
three dimensions, these results apply to finite element
discretizations of the spaces of vector fields, i.e., $H(\curl)$ and
$H(\div)$.
 
There is another special feature of the decomposition
\eqref{main-decomp} that we should point out.  For more standard
decompositions of the form \eqref{main-decomp}, derived from the
so-called degrees of freedom, cf. for example \cite[Theorem
  5.5]{bulletin}, only the macroelements corresponding to simplexes of
dimension greater or equal to $k$ will be utilized. For example, in
the three dimensional case, piecewise polynomials in $H\Lambda^1 \sim
H(\curl)$ can be decomposed into local contributions with support on
the macroelements associated to $f \in \Delta_m(\T)$, \, $1 \le m \le
3$, while for functions in $H\Lambda^2 \sim H(\div)$, we can restrict
to $2 \le m \le 3$. However, this restriction does not apply to the
decomposition \eqref{main-decomp} studied in this paper. In general,
we will utilize all values of $m$, $0 \le m \le n$, for all values of
$k$.
 
Extension operators have been a common tool used in many previous
papers to study the dependence of finite element methods with respect
to the polynomial degree. More precisely, properties of extensions
from a lower dimensional simplex to an $n$-simplex of the mesh are
utilized. This approach goes back to \cite{babuska-suri} in two space
dimensions, and with a simplification and an extension to three
dimensions given in \cite{munoz-sola}.  In particular, such operators
are used in \cite{additive-schwarz} for the construction of edge
bubbles in two space dimensions, and for the face bubbles in three
dimensions.  In contrast to this approach, such operators are not used
in the analysis given in this paper.  Here, we rely entirely on the
modification of the average operators represented by the trace
preserving operators $C_m^k$ defined by \eqref{Cm-rewritten}.  In the
case of zero forms, in previous papers, the bubbles are then
constructed by a sequential procedure, starting with vertex bubbles as
explained in Section~\ref{sec:scalar} above. This is the construction
utilized in \cite{bubble-I}, and it is closely related to procedures
used for the vertex bubbles and the edge bubbles in three dimensions
in \cite{additive-schwarz}. The construction done in \cite{K-M-R} on a
two dimensional boundary of a three dimensional domain also follows
this path.  However, in order to design a procedure that works well
also for differential forms of higher order, the construction done in
this paper deviates from the sequential procedure above, even in the
case of zero-forms.  We explained this alternative approach for
zero-forms in Section~\ref{sec:scalar}, where we show that even if the
individual operators $C_{m,f}^0$ do not preserve piecewise smoothness,
their sum will have this property. In fact, we study the difference
$C_m^0 - C_{m-1}^0$ and discover that this operator admits a
representation of the form
 \[
 (C_m^0 - C_{m-1}^0)u = \sum_{\substack{f \in \Delta_j\\ j= m,m-1}}
 K_{m,f}^0u,
 \]
cf. \eqref{decomp-0}, where the operators $K_{m,f}^0$ are local
operators with the support of $K_{m,f}^0u$ in $\Omega_f$.
Furthermore, a key observation is that the operators $K_{m,f}^0$
preserve piecewise smoothness due to an alternative representation,
derived by applying the boundary operator, $\partial_0$, to the
collection $\{\lambda_i - \frac{\rho_f}{|f^*|}\}_{x_i \in f^*}$. This
collection is seen as an element of the space $\C_0(f^*) \otimes
\P_1(\T)$, i.e., a space of $\P_1$-valued $0$-chains defined on the
link $f^*$.  This construction appears to be different from what has
been used before even in the case of zero forms, and motivates the
construction done for $k$-forms later in the paper. In particular, the
results obtained in Section~\ref{sec:mesh-functions} for the local
structure of the mesh, derived from the spaces $\C_j(f^*) \otimes
\P_1^-\Lambda^k(\T,f^*)$, are essential. These results are used to
construct the desired weight functions for the order reduction
operators, $R_{e,f}^k$, in Section~\ref{sec:order-reduct}, a key tool
used to define the desired local operators, $K_{m,f}^k$, in
Section~\ref{sec:Kmgk}. Finally, in Sections~\ref{sec:fund-ident} and
\ref{sec:bounds} we verify that the construction fulfills the results
stated in Theorem~\ref{thm:main}.

\bmhead{Acknowledgements}
The authors are grateful to Snorre H. Christiansen for helpful
discussions regarding the background material presented in
Section~\ref{sec:prelims}.

\bmhead{Funding}
The research leading to these results has received funding from the
European Research Council under the European Union's Seventh Framework
Programme (FP7/2007-2013) / ERC grant agreement 339643.

\bibliographystyle{amsplain}

\begin{thebibliography}{00}


\bibitem{FEEC-book} 
D.~N. Arnold, \emph{Finite Element Exterior
  Calculus}, CBMS-NSF Regional Conf. Ser. in Appl. Math. \textbf{93}, SIAM,
  Philadelphia, 2018. xi+120 pp. 

\bibitem{acta}
D.~N. Arnold, R.~S. Falk, and R. Winther, \emph{Finite element
  exterior calculus, homological techniques, and applications}, Acta Numerica
  \textbf{15} (2006), 1--155. 

\bibitem{bulletin}
  D.~N. Arnold, R.~S. Falk, and R. Winther,
  \emph{Finite element exterior calculus: from {H}odge theory to
  numerical stability}, Bull. Amer. Math. Soc. (N.S.) \textbf{47} (2010),
  no.~2, 281--354. 


\bibitem{babuska-suri}
  I. Babu$\check{s}$ka and M. Suri,
   \emph{The $h-p$ version of the finite element
     method with quasiuniform meshes}, ESAIM: Mod\'elisation
  math\'ematique et analyse num\'erique, vol. 21 (1987) no. 2, 199-238. 
 
\bibitem{daverman-sher}
  J.~L. Bryant, \emph{Piecewise linear topology}, in R. Daverman and R. Sher,
  eds, \emph{Handbook of Geometric Topology, Chapter 5}, 2001,
  North-Holland, Amsterdam.

\bibitem{bubble-I}
  R.~S. Falk and R. Winther,
\emph{The bubble transform: a new tool for analysis of finite element methods},
Found. Comput. Math., vol. 16 (2016), no. 1, 297-328.

\bibitem{bubble-II}
  R.~S. Falk and R. Winther, \emph{The bubble
    transform and the de Rham complex}, Found. Comput. Math., vol. 24 (2024),
  99-147.
  
\bibitem{K-M-R}
  M. Karkulik, J. Melenk, and A. Rieder, \emph{Stable decompostions of
    $hp$-BEM spaces and an optimal Schwarz preconditioner for the hypersingular
    integral operator in 3D}, ESAIM Math. Model. Numer. Anal. 54, no. 1 (2020),
  145-180.

\bibitem{munoz-sola}
  R. Mu\~noz Sola, \emph{Polynomial liftings on a tetrahedron and
    applications to the $h-p$ version of the finite element method
    in three dimensions}, SIAM J. Numer. Anal. 34, no. 1 (1997), 282-314.

\bibitem{nedelec-1980}
  J.~C. N\'ed\'elec, \emph{Mixed finite elements in $\R^3$}, Numer. Math.
  35, (1980), 315-341.

\bibitem{nedelec-1986}
  J.~C. N\'ed\'elec, \emph{A new family of mixed finite elements in $\R^3$},
  Numer. Math. 50, (1986), 57-81.

\bibitem{additive-schwarz} J. Sch\"{o}berl, J.~M. Melenk, C. Pechstein, and
  S. Zaglmayr, \emph{Additive Schwarz preconditioning for $p$-version
    triangular and tetrahedral finite elements},
IMA J. Numer. Anal. 28 (2008), no. 1, 1-24.

\bibitem{spanier}
  E.~H. Spanier, \emph{Algebraic Topology}, Springer-Verlag, New York,
  Corr. 3rd Edition, 1994.
  
\bibitem{tosseli-widlund}
  A. Toselli and O. Widlund, \emph{Domain Decomposition Methods -
    Algorithms and Theory}, volume 34 of
Springer Series in Computational Mathematics. Springer, 2005. 
\end{thebibliography}

\end{document}